\documentclass[a4paper]{amsproc}
\usepackage{amssymb}
\usepackage[hyphens]{url} 
\usepackage[dvips]{graphicx}
\usepackage{amsmath}
\usepackage{cite}
\usepackage{mathrsfs}
\usepackage[title]{appendix}
\usepackage[utf8]{inputenc}
\usepackage[english]{babel}
\usepackage{hyperref}
\usepackage{textgreek}
\usepackage{xcolor}

\hypersetup{linktoc=page}

 \theoremstyle{plain}
 \newtheorem{thm}{\bfseries Theorem}[section]
 \newtheorem{theorem}{\bfseries Theorem}
 \newtheorem{prop}[thm]{\bfseries Proposition}
 
 \newtheorem{lem}[thm]{\bfseries Lemma}
 \newtheorem{claim}[thm]{Claim}
 \newtheorem{cor}[thm]{\bfseries Corollary}
 \newtheorem{exm}[thm]{\bfseries Example}
 \newtheorem{dfn}[thm]{\bfseries Definition}
 \theoremstyle{remark}
 \newtheorem{rem}[thm]{Remark}
 \numberwithin{equation}{thm}
 \newenvironment{subproof}[1][\proofname]{%
  \begin{proof}[#1]%
}{%
  \end{proof}%
}

  \renewcommand{\leq}{\leqslant}
\renewcommand{\geq}{\geqslant}

\title[A.C.I.M AND KRIEGER-TYPE OF MARKOV SUBSHIFTS]
{On Absolutely Continuous Invariant Measures and Krieger-Type of Markov Subshifts}

\subjclass[2010]{37A20, 37A40, 60J10}

\keywords{non-singular transformations, non-homogeneous Markov chains, absolutely continuous invariant measure, Krieger-type}

\author[Nachi Avraham-Re'em]{\bfseries Nachi Avraham-Re'em}

\address{
Einstein Institute of Mathematics \\
Edmond J. Safra Campus (Givat-Ram)\\
The Hebrew University of Jerusalem\\
9190401\\
Israel}

\email{nachman.avraham@mail.huji.ac.il or nachi.avraham@gmail.com}

\begin{document}

\maketitle

\begin{abstract}
It is shown that for a nonsingular conservative shift on a topologically-mixing Markov subshift with the \textit{Doeblin condition} the only possible absolutely continuous shift-invariant measure is a Markov measure. Moreover, if it is not equivalent to a homogeneous Markov measure then the shift is of Krieger-type $\mathrm{III}_1$. A criteria for equivalence of Markov measures is included.
\end{abstract}

\tableofcontents

\section{Introduction and Main Theorems}

In recent years some general results were obtained about the classification of the Bernoulli shift according to its Krieger-type. The basic problem is classical and goes back to Halmos \cite{halmos1947invariant}: for a given sigma-finite Borel measure $\mu$ on a standard Borel space $X$ and a nonsingular Borel transformation $T:\left(X,\mu\right)\to\left(X,\mu\right)$, determine whether there exists a sigma-finite Borel measure $\nu$ which is both absolutely continuous with respect to $\mu$ and invariant to $T$. Such measure $\nu$ is abbreviated as \emph{a.c.i.m.} (absolutely continuous invariant measure) for $\mu$ and $T$. Hamachi \cite{hamachi1981} showed that there is a Bernoulli shift without a.c.i.m but he did not determine its Krieger-type. It was an open question, famously attributed to Krengel and Weiss \cite{krengel1970transformations, kosloff2014, danilenkosilva2008} (see also the MathSciNet review of Krengel on \cite{MR662470}), whether Krieger-types $\mathrm{II}_{\infty}$ and $\mathrm{III}_{\lambda}$ ($0\leq\lambda\leq1$) can appear in the nonsingular conservative shift. More details on the history of the problem can be found in the survey of Danilenko and Silva \cite{danilenkosilva2008}.

Only in the last few years some general results were discovered on the Bernoulli shift. First, Kosloff \cite{kosloff2014} showed that in the \textit{half-stationary} Bernoulli shift on two states space, when the distribution on all the negative coordinates is $\left(1/2,1/2\right)$, if the shift is nonsingular and conservative then it is either equivalent to a corresponding stationary Bernoulli measure, and then it is of Krieger-type $\mathrm{II}_1$, or that there is no any a.c.i.m. Moreover, in the latter case it is of Krieger-type $\mathrm{III}_1$. This result was later extended by Danilenko and Lemańczyk \cite{danilenko2019} when the distribution of the negative coordinates is $\left(p,1-p\right)$ for some $0<p<1$.

Recently, a significant progress has been achieved for Bernoulli actions of countable groups. Vaes and Wahl \cite{vaes2018bernoulli} formulated a characterization of a countable group to admit a Krieger-type $\mathrm{III}_1$ Bernoulli action in terms of the first $l^2$-cohomology of the group, and proved this characterization for a large family of groups. Björklund and Kosloff \cite{bjorklundKosloff2018} showed that every countable amenable group admits a Krieger-type $\mathrm{III}_1$ Bernoulli action on two states. The recent result of Björklund, Kosloff and Vaes \cite{bjorklund2020ergodicity} confirms the conjecture of Vaes and Wahl, showing that every countable group which is either amenable or has non-trivial first $l^2$-Betti number admits a Bernoulli action of Krieger-type $\mathrm{III}_1$.

In contrast to the Bernoulli shift, very few is known about Markov shift. The ergodicity of a nonsingular conservative Markov subshift was studied by Kosloff \cite{kosloff2019proving} and Danilenko \cite{danilenko2019weak} (see Theorem \ref{Theorem: Theorem A}). The Golden Mean Markov subshift model was used by Kosloff to construct examples of conservative Anosov diffeomorphisms of the torus $\mathbb{T}^2$ without a Lebesgue a.c.i.m. \cite{kosloff2018manifolds,kosloff2014conservativeanosov}. A special case of a half-stationary Markov shift was studied by Danilenko and Lemańczyk \cite{danilenko2019}, and they asked about a general half-stationary Markov shift on two states (see \cite[Problem~(1)]{danilenko2019}). Here we solve the Markov case to a relatively large extent under the \textit{Doeblin condition} and we remove the restrictive assumption of half-stationarity.

\vspace{5mm}

We now introduce our general setting. Let $X=\mathcal{S}^{\mathbb{Z}}$ for a finite state space $\mathcal{S}$ and consider the left-shift $T:X\to X$ defined by $\left(Tx\right)_{n}=x_{n+1}$ for every $n\in\mathbb{Z}$, where $x_n$ denotes the $n^{\text{th}}$ coordinate of $x$. For a $\left\{0,1\right\}$-valued $\left|\mathcal{S}\right|\times\left|\mathcal{S}\right|$-matrix $A$ let the subshift of finite type (SFT) associated to $A$ be the shift-invariant space
$$X_A=\left\{x\in X: A\left(x_n,x_{n+1}\right)=1\,\forall n\in\mathbb{Z}\right\}.$$
We call $A$ the \textit{adjacency matrix} of $X_A$. An SFT $X_A$ is called \textit{topologically-mixing} if $A$ is a primitive matrix, that is there exists a positive integer $M\geq1$ such that all the entries of $A^M$ are positive. Let $\left(X_n:n\in\mathbb{Z}\right)$ be the coordinates random variables of $X_A$, defined by $X_n\left(x\right)=x_n$. A \textit{Markov measure} $\mu$ on an SFT $X_A$ is a probability measure defined as follows. Take a sequence $\left(P_n:n\in\mathbb{Z}\right)$ of \textit{transition matrices}, which are stochastic $\mathcal{S}\times\mathcal{S}$-matrices with the property that $P_n\left(s,t\right)=0$ whenever $A\left(s,t\right)=0$ for all $n\in\mathbb{Z}$ and $s,t\in\mathcal{S}$. By \textit{stochastic matrix} we mean a matrix whose entries are non-negative and each of its rows is summed up to $1$. Take further a sequence $\left(\pi_n:n\in\mathbb{Z}\right)$ of probability distributions on $\mathcal{S}$ which, when relating them as row vectors, satisfy the identities
\begin{equation}
\label{eq:21}
\pi_nP_n=\pi_{n+1}\text{ for all }n\in\mathbb{Z}.
\end{equation}
This defines $\mu$ on cylinders via
$$\mu\left(X_{k+1}=s_{1},\dotsc,X_{k+m}=s_{m}\right)=\pi_{k+1}\left(s_{1}\right)P_{k+1}\left(s_{1},s_{2}\right)\dotsm P_{k+m-1}\left(s_{m-1},s_{m}\right),$$
for all $k\in\mathbb{Z}$, $m\in\mathbb{N}$ and $s_1,\dotsc,s_m\in\mathcal{S}$.
The consistency conditions \eqref{eq:21} ensures that this definition extends uniquely to a Borel measure $\mu$ on $X_A$ such that
$$\pi_n\left(s\right)=\mu\left(X_n=s\right),\quad n\in\mathbb{Z}, s\in\mathcal{S},$$
and
$$P_n\left(s,t\right)=\mu\left(X_{n+1}=t\mid X_n=s,X_{n-1}=s_1,\dotsc,X_{n-k}=s_k\right),$$
for all $t,s,s_1,\dotsc,s_k\in\mathcal{S}$ and $n\in\mathbb{Z}$. This last property is the usual Markov property. We write $\mu_{\left(P_n:n\in\mathbb{Z}\right)}$ for a Markov measure whose sequence of transition matrices is $\left(P_n:n\in\mathbb{Z}\right)$. Let us denote the \textit{reverse} transition matrices of a Markov measure $\mu_{\left(P_n:n\in\mathbb{Z}\right)}$ by
\begin{equation}
\label{eq:22}
\widehat{P}_{n}\left(s,t\right)=\mu\left(X_{n-1}=t\mid X_{n}=s\right)=\frac{\pi_{n-1}\left(t\right)}{\pi_{n}\left(s\right)}P_{n-1}\left(t,s\right),
\end{equation}
for all $n\in\mathbb{Z}$ and $s,t\in\mathcal{S}$ with $\pi_n\left(s\right)>0$. When for some $n\in\mathbb{Z}$ and $s\in\mathcal{S}$ we have $\pi_n\left(s\right)=0$, we let $\widehat{P}_n\left(s,t\right)=0$ for all $t\in\mathcal{S}$. We extend the notation $\widehat{Q}$ also for an $\mathcal{S}\times\mathcal{S}$-stochastic matrix $Q$, by relating it as a constant sequence of transition matrices, which together with the distribution $\lambda$ on $\mathcal{S}$ satisfying $\lambda Q=\lambda$ defines a Markov measure on $\mathcal{S}^{\mathbb{Z}}$.

The following condition of a Markov measure is fundamental to our work. We call it \textit{Doeblin condition} after various conditions of this type formulated by W. Döblin \cite{cohn1981}. Let $\mu=\mu_{\left(P_n:n\in\mathbb{Z}\right)}$ be a Markov measure on an SFT $X_A$. We say that $\mu$ satisfies the Doeblin condition if
\begin{equation}
\label{Doeblin}\tag{$\boldsymbol{D}$}
\exists\delta>0, P_{n}\left(s,t\right)\geq\delta\iff A\left(s,t\right)=1\text{ for all }s,t\in\mathcal{S}\text{ and }n\in\mathbb{Z}.
\end{equation}\

The following result was proved in \cite[Proposition~2.2, Theorem~3.4]{kosloff2019proving}.

\begin{theorem}[Kosloff]
\label{Theorem: Theorem A}
Let $X_A\subset\mathcal{S}^{\mathbb{Z}}$ be a topologically-mixing SFT of $\mathcal{S}^{\mathbb{Z}}$ and $\mu$ be a Markov measure on $X_A$ with the Doeblin condition \ref{Doeblin}. Suppose that the shift is nonsingular with respect to $\mu$. Then if the shift is conservative it is ergodic.
\end{theorem}

Danilenko \cite{danilenko2019weak} strengthened this result and showed that in this case the shift is further weakly-mixing, in the sense that its product with every ergodic probability measure preserving transformation is again ergodic.

Our work goes further into the classification of the shift acting on a Markov subshift of finite type into its Krieger-type. Note that by the Poincaré recurrence theorem when the shift is not conservative it can not admit a shift-invariant probability measure. Then in the following Theorems \ref{Theorem: Theorem B}, 
\ref{Theorem: Theorem C}, \ref{Theorem: Theorem D} and \ref{Theorem: Theorem E} we assume that the shift is nonsingular and conservative with respect to the subject measure. See the exact definitions below.

In the case that we call \textit{the divergent scenario} the following theorem fully answers the question of possible Krieger-type of the shift.

\begin{theorem}
\label{Theorem: Theorem B}
Let $X_A\subset\mathcal{S}^{\mathbb{Z}}$ be a topologically-mixing SFT and $\mu=\mu_{\left(P_n:n\in\mathbb{Z}\right)}$ be a Markov measure on $X_A$ with the Doeblin condition \ref{Doeblin}. If the shift is nonsingular and conservative with respect to $\mu$ and the limit $\lim_{\left|n\right|\to\infty}P_{n}$ does not exist, the shift is of Krieger-type $\mathrm{III}_1$.
\end{theorem}
By relating to the limit $\lim_{\left|n\right|\to\infty}P_{n}$ we mean that it exists if, and only if, the limits $\lim_{n\to\infty}P_{n}$ and $\lim_{n\to-\infty}P_{n}$ both exist entrywise and are equal.

The other case that we call \textit{the convergent scenario} is more subtle. We first give a necessary criteria for the conservativeness of the shift. This condition was established in \cite[Lemma~8.6]{danilenko2019} in the special case of half-stationary bistochastic two states case, and here we establish the general case of topologically-mixing Markov SFT with the Doeblin condition.

\begin{theorem}
\label{Theorem: Theorem C}
Let $X_A\subset\mathcal{S}^{\mathbb{Z}}$ be a topologically-mixing SFT and $\mu=\mu_{\left(P_n:n\in\mathbb{Z}\right)}$ be a Markov measure on $X_A$ with the Doeblin condition \ref{Doeblin}. If the shift is nonsingular and conservative with respect to $\mu$ and both limits $\lim_{n\to\infty}P_{n}$ and $\lim_{n\to\infty}P_{-n}$ exist, then they are equal. That is, $\lim_{\left|n\right|\to\infty}P_{n}$ exists.
\end{theorem}

Then we can determine the Krieger-type of the shift as follows.

\begin{theorem}
\label{Theorem: Theorem D}
Consider the state space $\mathcal{S}=\left\{0,1\right\}$. Let $X_A\subset\mathcal{S}^{\mathbb{Z}}$ be a topologically-mixing SFT and $\mu=\mu_{\left(P_n:n\in\mathbb{Z}\right)}$ be a Markov measure on $X_A$ with the Doeblin condition \ref{Doeblin}. If the shift is nonsingular and conservative with respect to $\mu$, the Krieger-type of the shift is either $\mathrm{II}_1$ or $\mathrm{III}_1$.\\
Moreover, the shift is of Krieger-type $\mathrm{II}_1$ if, and only if, there exists a stochastic matrix $Q$ such that 
$Q=\lim_{\left|n\right|\to\infty}P_{n}$ and
$$\sum_{n\geq1}\sum_{s,u,v,t\in\mathcal{S}}\left(\sqrt{\widehat{P}_{-n}\left(u,s\right)P_{n}\left(v,t\right)}-\sqrt{\widehat{Q}\left(u,s\right)Q\left(v,t\right)}\right)^{2}<\infty.$$
In this case, the absolutely continuous invariant measure for the shift is the Markov measure defined by $Q$ and the distribution $\lambda$ on $\mathcal{S}$ satisfying $\lambda Q=\lambda$.
\end{theorem}

Consider the Golden Mean SFT $X_{\boldsymbol{G}}\subset\left\{0,1,2\right\}^{\mathbb{Z}}$ that is defined by the primitive adjacency matrix
$$\boldsymbol{G}=\left(\begin{array}{ccc}
1 & 0 & 1\\
1 & 0 & 1\\
0 & 1 & 0
\end{array}\right).$$

\begin{theorem}
\label{Theorem: Theorem E}
Let $X_{\boldsymbol{G}}\subset\left\{0,1,2\right\}^{\mathbb{Z}}$ be the Golden Mean SFT and $\mu=\mu_{\left(P_n:n\in\mathbb{Z}\right)}$ be a Markov measure on $X_{\boldsymbol{G}}$ with the Doeblin condition \ref{Doeblin}. If the shift is nonsingular and conservative with respect to $\mu$, the Krieger-type of the shift is either $\mathrm{II}_1$ or $\mathrm{III}_1$. These alternatives are determined by the same test of Theorem \ref{Theorem: Theorem D}.
\end{theorem}

\subsection{About the Proof}

In the first works \cite{kosloff2014,danilenko2019} the authors proved the case of a half-stationary Bernoulli shift by computing the ratio set of the shift using the appropriate cocycle. However, this method relies on the Bernoullicity and the half-stationarity of the shift, and in the general Markov case we found the computation of the essential values of this cocycle to be more involved. In later works \cite{kosloff2019proving,danilenko2019weak,bjorklund2020ergodicity} it has been found useful to study the ergodicity of the shift by the action of the permutations that change only finitely many coordinates. This approach applied by Björklund, Kosloff and Vaes \cite{bjorklund2020ergodicity} for amenable groups, using a ratio ergodic theorem by Danilenko \cite{danilenko2019weak}, to replace the computation of the ratio set of the Bernoulli shift by the computation of the ratio set of the finite permutations action. However, also in this approach the Bernoullicity plays a crucial role in two aspects. The first is that the cocycle of the shift satisfies a special identity with respect to finite permutations (see \cite[Lemma~3.1]{bjorklund2020ergodicity}) and this identity no longer holds in the Markov case. The second is that the finite permutations action is ergodic with respect to Bernoulli measures. This is far from being true in general and in the Markov case it is not true even when the shift is measure-preserving. See Example 3 of Blackwell--Freedman \cite{blackwell1964}. In particular, the action of the finite permutations when is not ergodic does not fall under the Krieger classification.

Here we place the above approach for amenable groups in a more general context. We develop a notion of \textit{Renormalization Full-Group} (Definition \ref{Definition: Renormalization Full-Group}) of one action of countable group with respect to another action of countable group, where the latter satisfies a metric property with respect to the former. This metric property is the Maharam extension-version of the notion of equivalence underlying the well-known Hopf Argument. We then establish a version of Hopf Argument for the Maharam extension (Theorem \ref{Theorem: Hopf Argument}), which allows one to study the ratio set of the first action by the ratio set of the corresponding renormalization full-group action. Our use of this renormalization process can be viewed, in a sense, as replacing the computation of the ratio set of groups with a notion of past and future, like the shift, with the computation of the ratio set of some symmetry group.

\subsection*{Acknowledgement}

I would like to express my deep gratitude to my advisor, Zemer Kosloff, for his patient, generous guidance and for the help in this research. Many important insights in this work are inspired by oral discussions with him. I also want to thank the anonymous referee for their careful reading and for many valuable suggestions that improved the paper.

\section{Preliminaries}

In this work all the measurable spaces are standard Borel spaces and all the measures are Borel and sigma-finite. Two measures $\nu$ and $\mu$ on a standard Borel space $X$ are called \textit{equivalent} if each of $\nu$ and $\mu$ is absolutely continuous with respect to the other. An \textit{automorphism} of a measurable space $\left(X,\mu\right)$ is a bi-measurable invertible transformation $V$ of $X$ onto $X$, which is \textit{nonsingular} with respect to $\mu$; that is, $\mu$ and $\mu\circ V^{-1}$ are equivalent measures. The automorphisms group of $\left(X,\mu\right)$ is denoted by $\mathrm{Aut}\left(X,\mu\right)$. When there is no confusion we write
$$V'\left(x\right)=\frac{d\mu\circ V}{d\mu}\left(x\right)\in L^1\left(X,\mu\right),\quad V\in\mathrm{Aut}\left(X,\mu\right).$$
Let $\Gamma$ be a countable group. We write $\Gamma\curvearrowright\left(X,\mu\right)$ for a group homomorphism $T:\Gamma\to\mathrm{Aut}\left(X,\mu\right)$. When there is no confusion we write $\gamma x$ for $T\left(\gamma\right)\left(x\right)$. Such action is called \textit{ergodic} if for every Borel set $E\subset X$, if $\gamma E\subset E$ for all $\gamma\in\Gamma$ then either $\mu\left(E\right)=0$ or $\mu\left(X\backslash E\right)=0$. It is called \textit{conservative} if for every Borel set $E\subset X$ with $\mu\left(E\right)>0$ there exists $\gamma\in\Gamma$ not the identity with $\mu\left(E\cap\gamma E\right)>0$. Note that for a non-atomic measure, ergodicity is stronger then conservativeness. Also note that nonsingularity, ergodicity and conservativeness are invariant properties under equivalence of measures.

Let $\left(X,\mu\right)$ be a nonatomic standard measure space and $\Gamma\curvearrowright\left(X,\mu\right)$ be a nonsingular ergodic action. Suppose that there exists a measure $\nu$ on $X$ which is both absolutely continuous with respect to $\mu$ and invariant under the action $\Gamma\curvearrowright\left(X,\nu\right)$. Such measure $\nu$ is called a.c.i.m. (absolutely continuous invariant measure) for $\Gamma\curvearrowright\left(X,\mu\right)$. In that case the action is said to be of Krieger-type $\mathrm{II}_1$ or of Krieger-type $\mathrm{II}_{\infty}$, depending on whether its a.c.i.m. is finite or infinite (this does not depend on the choice of the a.c.i.m. by the ergodicity). If the action does not admit an a.c.i.m. it is said to be of Krieger-type $\mathrm{III}$.

\subsection*{The Full-Group, Orbital Cocycles and Essential Values}

A Borel equivalence relation $\mathcal{R}$ is a Borel subset of $X\times X$ for which $x\sim y\iff \left(x,y\right)\in\mathcal{R}$ is an equivalence relation. For a Borel set $E\subset X$ we write $\mathcal{R}\left(E\right)$ for the $\mathcal{R}$-saturation $\left\{y\in X:\exists x\in E, \left(x,y\right)\in\mathcal{R}\right\}$ of $E$. For $x\in X$ write $\mathcal{R}\left(x\right)$ for $\mathcal{R}\left(\left\{x\right\}\right)$.

Such $\mathcal{R}$ is called \textit{countable} if $\mathcal{R}\left(x\right)$ is a countable set for $\mu$-almost every $x\in X$. It is called \textit{nonsingular} if $\mu\left(\mathcal{R}\left(E\right)\right)=0$ whenever $\mu\left(E\right)=0$. A fundamental type of Borel countable equivalence relation is the \textit{orbital equivalence relation} $\mathcal{O}_{\Gamma}$ of a countable group action $\Gamma\curvearrowright\left(X,\mu\right)$. This equivalence relation consists of all $\left(x,\gamma x\right)$ for $x\in X$ and $\gamma\in\Gamma$. By the Feldman--Moore Theorem \cite{feldman1977} every nonsingular countable Borel equivalence relation $\mathcal{R}$ is the orbital equivalence relation of some (non-unique) countable group of automorphisms $\mathrm{FM}\left(\mathcal{R}\right)\curvearrowright\left(X,\mu\right)$.

The \textit{full-group} $\left[\mathcal{R}\right]$ of $\mathcal{R}$ consists of all $V\in\mathrm{Aut}\left(X,\mu\right)$ such that $\left(x,Vx\right)\in\mathcal{R}$ for $\mu$-almost every $x\in X$. The \textit{pseudo full-group} $\left[\left[\mathcal{R}\right]\right]$ of $\mathcal{R}$ consists of all nonsingular one-to-one Borel transformations $V:D\to V\left(D\right)$ for some Borel domain $D\subset X$, such that $\left(x,Vx\right)\in\mathcal{R}$ for $\mu$-almost every $x\in D$. We write $\left[\Gamma\right]$ and $\left[\left[\Gamma\right]\right]$ for $\left[\mathcal{O}_{\Gamma}\right]$ and $\left[\left[\mathcal{O}_{\Gamma}\right]\right]$, respectively. An \textit{orbital cocycle}, or                                                                          simply \textit{cocycle}, for a Borel equivalence relation $\mathcal{R}$ is a function $\varphi:\mathcal{R}\to\mathbb{R}$ for which there exists $X_0\subset X$ of $\mu$-full measure such that for all $\left(x,y\right),\left(y,z\right)\in\left(X_0\times X_0\right)\cap\mathcal{R}$ it holds that
$$\varphi\left(x,z\right)=\varphi\left(x,y\right)+\varphi\left(y,z\right).$$
We write $\varphi_V\left(x\right)=\varphi\left(x,Vx\right)$ for every $V\in\left[\left[\mathcal{R}\right]\right]$ and $x\in X_0$. For a nonsingular Borel equivalence relation $\mathcal{R}$ on $\left(X,\mu\right)$ there is a fundamental orbital cocycle called the (log) \textit{Radon--Nikodym cocycle}. This can be defined for every choice of $\Gamma=\mathrm{FM}\left(\mathcal{R}\right)$ by
$$\varphi_{\gamma}\left(x\right)=\log\frac{d\mu\circ\gamma}{d\mu}\left(x\right)\in L^1\left(X,\mu\right),\quad\gamma\in\Gamma.$$
This definition does not depend on the choice of $\mathrm{FM}\left(\mathcal{R}\right)$ up to a $\mu$-null set.

A number $r\in\mathbb{R}$ is called an \textit{essential value} for $\Gamma\curvearrowright\left(X,\mu\right)$, if for every Borel set $E$ with $\mu\left(E\right)>0$ and every $\epsilon>0$ there exists $V\in\left[\left[\Gamma\right]\right]$ such that 
$$\mu\left(E\cap V^{-1}E\cap\left\{\left|\varphi_{V}-r\right|<\epsilon\right\}\right)>0.$$

The following lemma is useful to compute essential values. It can be found in several formulations in \cite[Lemma~2.1]{choksi1987}\cite[Lemma~1.1]{danilenko2019}\cite[Lemma~7]{kosloff2018manifolds}.

\begin{lem}
\label{Lemma: Approximating Essential Values}
Let $\Gamma\curvearrowright\left(X,\mu\right)$ be a countable group of automorphisms and let $\varphi$ be its Radon--Nikodym cocycle. Let $\mathcal{C}$ be a $\mu$-dense countable algebra in the Borel sigma-algebra. Then a number $r\in\mathbb{R}$ is an essential value for $\Gamma\curvearrowright\left(X,\mu\right)$ if there exists $\eta>0$ depending only on $r$, such that the following condition holds.
\begin{quote}
For every $\epsilon>0$ and every $C\in\mathcal{C}$ with $\mu\left(C\right)>0$ there exists $F\subset C$ and $V\in\left[\left[\Gamma\right]\right]$, such that $V:F\to V\left(F\right)\subset C$ and $\mu\left(F\right)\geq\eta\mu\left(C\right)$ and $\left|\varphi_{V}\left(x\right)-r\right|<\epsilon$ for all $x\in F$.
\end{quote}
\end{lem}

\subsection*{Krieger's Ratio Set and The Maharam Extension}

The collection of all essential values for the Radon--Nikodym cocycle of $\Gamma\curvearrowright\left(X,\mu\right)$ is called \textit{the Krieger ratio set} following \cite{krieger1970araki} (see also Schmidt's monograph \cite[Chapter~3]{schmidt1977}), or simply \textit{the ratio set}, and is denoted by $\mathrm{e}\left(\Gamma,\mu\right)$. When there is no confusion we write $\mathrm{e}\left(\Gamma\right)$ for $\mathrm{e}\left(\Gamma,\mu\right)$. Observe that $\mathrm{e}\left(\Gamma,\nu\right)=\mathrm{e}\left(\Gamma,\mu\right)$ whenever $\nu$ and $\mu$ are equivalent measures. It is well-known that the ratio set is not empty if, and only if, the action is conservative, and that the ratio set is a closed additive subgroup of $\mathbb{R}$. Hence, the ratio set of a conservative action is one of the following:
$$\left\{0\right\},\,\mathbb{R},\text{ or }\left\{n\log\lambda:n\in\mathbb{Z}\right\}\text{ for some }0<\lambda<1.$$
The ratio set has been defined by Krieger in order to classify nonsingular ergodic actions of type $\mathrm{III}$ into types $\mathrm{III}_{\lambda}$, $0\leq\lambda\leq 1$ as follows: Type $\mathrm{III}_0$ corresponds to ratio set that contains, in an appropriate sense, infinite values, and we do not deal with this here; type $\mathrm{III}_1$ corresponds to ratio set $\mathrm{e}\left(\Gamma,\varphi\right)=\mathbb{R}$; and, type $\mathrm{III}_{\lambda}$ for $0<\lambda<1$ corresponds to ratio set $\mathrm{e}\left(\Gamma,\varphi\right)=\left\{n\log\lambda:n\in\mathbb{Z}\right\}$ for $0<\lambda<1$, respectively. For more information on the ratio set and its role as an invariant of orbital equivalence we refer to \cite{hamachi1975,hamachi1981,katznelson1991classification}.

Let $\Gamma\curvearrowright\left(X,\mu\right)$ be a countable group of automorphisms. Consider the space $\widetilde{X}=X\times\mathbb{R}$ with the measure $d\widetilde{\mu}\left(x,t\right)=d\mu\left(x\right)\mathrm{exp}\left(t\right)dt$. The \textit{Maharam extension} of $\Gamma\curvearrowright\left(X,\mu\right)$ is the action of $\Gamma$ on $\left(\widetilde{X},\widetilde{\mu}\right)$ defined by
$$\widetilde{\gamma}\left(x,t\right):=\left(\gamma x,t-\log\frac{d\mu\circ\gamma}{d\mu}\left(x\right)\right),\quad\gamma\in\Gamma.$$
The Maharam extension is an infinite sigma-finite measure-preserving action and we denote this action by $\widetilde{\Gamma}\curvearrowright\left(\widetilde{X},\widetilde{\mu}\right)$. By a well-known theorem of Maharam (for transformations) \cite{maharam1964} \cite[Chapter~3.4]{aaronson1997} and Schmidt (for general countable groups) \cite[Theorem~5.5]{schmidt1977}, the Maharam extension of a conservative action is conservative. The Maharam extension of an ergodic countable group of automorphisms $\Gamma\curvearrowright\left(X,\mu\right)$ is itself ergodic if, and only if, $\Gamma\curvearrowright\left(X,\mu\right)$ is of type $\mathrm{III}_1$ \cite[Corollary~5.4]{schmidt1977}, \cite[Corollary~8.2.5]{aaronson1997}.

\section{Notations and Asymptotic Symbols}

We use the following common notations and abbreviations. The function $\mathrm{sign}\left(x\right)$ is $+1$ if $x$ is a non-negative number and $-1$ if $x$ is a negative number. For a random variable $Y$ with distribution $\mu$ we write $\mathbf{E}_{\mu}\left(Y\right)$ for its mean and $\mathbf{V}_{\mu}\left(Y\right)$ for its variance. We abbreviate the mean by $\mathbf{E}\left(Y\right)$ and the variance by $\mathbf{V}\left(Y\right)$ when there is no confusion. For an SFT $X_A$ and a set $I\subset\mathbb{Z}$ we write $\sigma\left(X_n:n\in I\right)$ for the sigma-algebra generated by cylinders supported on the coordinates of $I$. The operation $\ast$ will be used for concatenation of finite sequences as follows. For finite sequences $B=\left(b_1,\dotsc,b_L\right)$ and $B'=\left(b'_1,\dotsc,b'_{L'}\right)$ we let $B\ast B'$ be the finite sequence $\left(b_1,\dotsc,b_L,b'_1,\dotsc,b'_{L'}\right)$.

For a sequence $\left(a_n:n\geq1\right)$ in a metric space $\mathbb{M}$, we denote by $\mathcal{L}\left(a_{n}:n\geq1\right)$ the set of all partial limits of $\left(a_n:n\geq1\right)$ in $\mathbb{M}$.

We use asymptotic symbols similar to the Vinogradov notations as follows. For sequences $\left(a_{n}:n\geq1\right)$ and $\left(b_{n}:n\geq1\right)$ of numbers write
$$a_{n}\preccurlyeq b_{n}\iff\exists C>0\text{ with }\left|a_{n}\right|\leq C\left|b_{n}\right|\text{ for all }n\geq1.$$
Write also
$$a_{n}\asymp b_{n}\iff a_{n}\preccurlyeq b_{n}\text{ and }b_{n}\preccurlyeq a_{n}.$$
This defines an equivalence relation on sequences of numbers.

We use extensively the basic approximation
$$\frac{a-b}{a}<\log\left(a/b\right)<\frac{a-b}{b}\quad\text{for all }a,b>0.$$
Restricting ourselves to numbers in an interval $\left[c,C\right]$ for $0<c<C<\infty$, one can derive that for sequences $\left(a_{n}:n\geq1\right)$ and $\left(b_{n}:n\geq1\right)$ contained in $\left[c,C\right]$,
\begin{equation}
\label{Fact: Log Approximation}
\log\left(a_{n}/b_{n}\right)\asymp a_{n}-b_{n}.
\end{equation}
In particular, if for all $n\geq1$ we have $c\leq b_n\leq a_n\leq C$ then
$$\sum_{n\geq1}\log\left(a_{n}/b_{n}\right)=\infty\iff\sum_{n\geq1}\left(a_{n}-b_{n}\right)=\infty.$$

\section{Renormalization and The Hopf Argument}

Let $\left(X,\mu\right)$ be a standard measure space, $G$ be a countable group and $T_G\curvearrowright\left(X,\mu\right)$ be an action of $G$ by automorphisms. Fix a metric $\mathrm{d}$ on $X$ that induces its standard Borel structure. For a Borel countable equivalence relation $\mathcal{R}\subset X\times X$ we say that an element $\left(x,y\right)\in\mathcal{R}$ is an \textit{asymptotic pair} for $T_G$ if
$$\mathrm{d}\left(T_{g}\left(x\right),T_{g}\left(y\right)\right)\xrightarrow[g\to\infty]{}0.$$
The collection $\mathcal{R}\left(T_G\right)\subset\mathcal{R}$ of all asymptotic pairs for $T_{G}$ is a Borel sub-equivalence relation of $\mathcal{R}$ and in particular it is countable. Note that this is the notion of equivalence underlying the well-known Hopf Argument \cite{coudene2007hopf}. The term "asymptotic pair" is the common name for the analogous notion in topological dynamics \cite{blanchard2002li,chung2015homoclinic}.

Let $\Gamma\curvearrowright\left(X,\mu\right)$ be another countable group of automorphisms. We say that $\Gamma$ is \textit{asymptotic} for $T_G$ if $\left(x,\gamma x\right)\in\mathcal{O}_{\Gamma}$ is an asymptotic pair for $T_G$ for $\mu$-a.e. $x\in X$ and every $\gamma\in\Gamma$.

\begin{dfn} (Renormalization Full-Group)
\label{Definition: Renormalization Full-Group}
Let $T_G$ and $\Gamma$ be actions of countable groups of automorphisms and suppose that $\Gamma$ is asymptotic for $T_G$. Consider the Maharam extensions $\widetilde{T_{G}}\curvearrowright\left(\widetilde{X},\widetilde{\mu}\right)$ and $\widetilde{\Gamma}\curvearrowright\left(\widetilde{X},\widetilde{\mu}\right)$. Define the \emph{renormalization full-group} of $T_G$ with respect to $\Gamma$ to be
$$\mathscr{R}\left(T_{G};\Gamma\right):=\left\{ V\in\left[\Gamma\right]:\left(\left(x,t\right),\widetilde{V}\left(x,t\right)\right)\in\mathcal{O}_{\widetilde{\Gamma}}\left(\widetilde{T_G}\right)\text{ for }\widetilde{\mu}\text{-a.e. }\left(x,t\right)\in\widetilde{X}\right\},$$
where we take the metric on $\widetilde{X}$ to be the product of the chosen metric $\mathrm{d}$ on $X$ with the standard distance on $\mathbb{R}$. In a similar manner, we define the \emph{renormalization pseudo full-group} to be the set of all elements in $\left[\left[\Gamma\right]\right]$ satisfying the same condition as in the renormalization full-group. We will abbreviate the renormalization pseudo full-group by $\mathscr{R}\left(T_G;\left[\left[\Gamma\right]\right]\right)$.
\end{dfn}

There is a simple description of this object. As $\Gamma$ is asymptotic for $T_G$ the equivalence relation $\mathcal{O}_{\widetilde{\Gamma}}\left(\widetilde{T_G}\right)\subset\mathcal{O}_{\widetilde{\Gamma}}$ consists of all $\left(\left(x,t\right),\widetilde{V}\left(x,t\right)\right)\in\mathcal{O}_{\widetilde{\Gamma}}$ for some $V\in\left[\Gamma\right]$ such that
\begin{align*}
&\left(t-\log\left(T_{g}\right)'\left(x\right)\right)-\left(t-\log V\left(x\right)-\log\left(T_{g}\right)'\left(Vx\right)\right)\\
&\qquad\qquad\qquad\qquad\qquad\qquad\qquad=\log\frac{\left(T_{g}\circ V\right)'\left(x\right)}{\left(T_{g}\right)'\left(x\right)}\xrightarrow[g\to\infty]{}0.
\end{align*}
The occurrence of this condition does not depend on the second variable $t$. We then see that an element $V\in\left[\Gamma\right]$ belongs to $\mathscr{R}\left(T_G;\Gamma\right)$ if, and only if,
$$\frac{\left(T_{g}\circ V\right)'\left(x\right)}{\left(T_{g}\right)'\left(x\right)}\xrightarrow[g\to\infty]{}1\text{ for }\mu\text{-a.e. }x\in X.$$
To see that $\mathscr{R}\left(T_G;\Gamma\right)$ is a subgroup, recall that by the chain rule
$$\frac{\left(T_g\circ V^{-1}\right)'\left(x\right)}{\left(T_g\right)'\left(x\right)}=\frac{\left(T_g\right)'\left(V^{-1}x\right)}{\left(T_g\circ V\right)'\left(V^{-1}x\right)},\quad g\in G, V\in\left[\Gamma\right],$$
and
$$\frac{\left(T_g\circ V\circ W\right)'\left(x\right)}{\left(T_g\right)'\left(x\right)}=\frac{\left(T_g\circ V\right)'\left(Wx\right)}{\left(T_g\right)'\left(Wx\right)}\frac{\left(T_g\circ W\right)'\left(x\right)}{\left(T_g\right)'\left(x\right)},\quad g\in G, V,W\in\left[\Gamma\right].$$
This also shows that the renormalization pseudo full-group is a pseudo group.

The renormalization full-group may be uncountable, but its orbital equivalence relation $\mathcal{O}_{\mathscr{R}\left(T_G;\Gamma\right)}$ is countable as a sub relation of $\mathcal{O}_{\Gamma}$. It is also Borel since
\begin{align}
\label{eq:25}
\mathcal{O}_{\mathscr{R}\left(T_G;\Gamma\right)}=\bigcup_{\gamma\in\Gamma}O_{\gamma},
\end{align}
where $O_{\gamma}$ for $\gamma\in\Gamma$ is the set of all $\left(x,\gamma x\right)\in X\times X$ for which
$$\frac{\left(T_{g}\circ\gamma\right)'\left(x\right)}{\left(T_{g}\right)'\left(x\right)}\xrightarrow[g\to\infty]{}1.$$
Then by the Feldman--Moore Theorem we can consider the ratio set $\mathrm{e}\left(\mathscr{R}\left(T_G;\Gamma\right)\right)$.

\begin{exm}[Bernoulli Shift]
\label{Example:BernoulliShift}
Let $X=\mathcal{S}^{\mathbb{Z}}$ for some finite set $\mathcal{S}$ and suppose that $G=\mathbb{Z}$ acting by the shift $T$ and that $\Gamma=\Pi$ is the group of of all permutations of $\mathbb{Z}$ that change only finitely many elements. This group acts naturally on $\left(X,\mu\right)$ by letting $\pi\in\Pi$ be the automorphism defined by $\left(\pi x\right)_n=x_{\pi\left(n\right)}$ for all $n\in\mathbb{Z}$ and $x\in X$. It is clear that $\Gamma$ is asymptotic for $T$ for the metric
$$\mathrm{d}\left(x,y\right)=2^{-\inf\left\{n\in\mathbb{N}:x_n\neq y_n\right\}}$$
on $X$. Consider a product measure $\mu=\prod_{n\in\mathbb{Z}}\mu_{n}$ on $X$ and suppose that $\mu$ satisfies the Doeblin condition \ref{Doeblin}. We claim that if the shift is nonsingular with respect to $\mu$ then the renormalization full-group $\mathscr{R}\left(T;\Pi\right)$ of the shift with respect to the finite permutations is $\left[\Pi\right]$ itself. First note that the shift satisfies
$$\left(T^{n}\right)'\left(x\right)=\prod_{k\in\mathbb{Z}}\frac{\mu_{k-n}\left(x_{k}\right)}{\mu_{k}\left(x_{k}\right)},\quad n\in\mathbb{Z}.$$
Let $V:=V_{a,b}\in\left[\Pi\right]$ for some $a,b\in\mathbb{Z}$ be the transposition defined by $\left(Vx\right)_{a}=x_{b}$, $\left(Vx\right)_{b}=x_{a}$ and $\left(Vx\right)_{k}=x_{k}$ for any other $k\in\mathbb{Z}$. Then we have the formula
$$V'\left(x\right)=\frac{\mu_{a}\left(x_{b}\right)\mu_{b}\left(x_{a}\right)}{\mu_{a}\left(x_{a}\right)\mu_{b}\left(x_{b}\right)},$$
so one can see that
$$\frac{\left(T^{n}\circ V\right)'\left(x\right)}{\left(T^{n}\right)'\left(x\right)}=\frac{\left(T^{n}\right)'\left(Vx\right)}{\left(T^{n}\right)'\left(x\right)}V'\left(x\right)=\frac{\mu_{a-n}\left(x_{b}\right)}{\mu_{a-n}\left(x_{a}\right)}\frac{\mu_{b-n}\left(x_{a}\right)}{\mu_{b-n}\left(x_{b}\right)}.$$
Assuming that the shift is nonsingular with respect to $\mu$, we later see in Corollary \ref{Corollary: Nonsingularity of the Shift} that it satisfies $\mu_{n}\left(s\right)-\mu_{n-1}\left(s\right)\xrightarrow[\left|n\right|\to\infty]{}0$ for all $s\in\mathcal{S}$. Using the Doeblin condition we conclude that $\mu_{n}\left(s\right)/\mu_{n-1}\left(s\right)\xrightarrow[\left|n\right|\to\infty]{}1$ for all $s\in\mathcal{S}$ which implies that $\mu_{a-n}\left(s\right)/\mu_{b-n}\left(s\right)\xrightarrow[\left|n\right|\to\infty]{}1$ for fixed $a,b\in\mathbb{Z}$. This shows that $V\in\mathscr{R}\left(T;\Pi\right)$.
\end{exm}

\begin{thm}[The Hopf Argument for the Maharam Extension]\
\label{Theorem: Hopf Argument}
Let $\left(X,\mu\right)$ be a standard Borel probability space. Let $G$ be a countable \emph{amenable} group and let $T_G\curvearrowright\left(X,\mu\right)$ be a \emph{conservative} action of $G$ by automorphisms. Let $\Gamma\curvearrowright\left(X,\mu\right)$ be a countable group of automorphisms and assume that $\Gamma$ is asymptotic for $T_G$.\\
If the renormalization full-group satisfies $\mathrm{e}\left(\mathscr{R}\left(T_G;\Gamma\right)\right)=\mathbb{R}$ then also $\mathrm{e}\left(T_G\right)=\mathbb{R}$.\\
In particular, if $T_G$ is ergodic and $\mathrm{e}\left(\mathscr{R}\left(T_G;\Gamma\right)\right)=\mathbb{R}$ then $T_G$ is of type $\mathrm{III}_{1}$.
\end{thm}

\begin{rem}
The renormalization process can be described in terms of equivalence relations using \eqref{eq:25} in a similar way introduced by Danilenko in \cite{danilenko2019weak}. There can be found a translation of the following Lemma \ref{Lemma: Renormalization 1} to \cite[Theorem~2.3]{danilenko2019weak} under appropriate assumptions, by passing to an equivalent probability measure. For the sake of completeness we present here a self-contained proof.
\end{rem}

We formulate two lemmas that together imply Theorem \ref{Theorem: Hopf Argument}. The first lemma is a refinement of \cite[Lemma 3.1]{bjorklund2020ergodicity}.

\begin{lem}
\label{Lemma: Renormalization 1}
Let $T_G$ and $\Gamma$ be as in Theorem \ref{Theorem: Hopf Argument}. Then every $\widetilde{T_G}$-invariant $L^1\left(\widetilde{X},\widetilde{\mu}\right)$-function is also $\widetilde{\mathscr{R}\left(T_G;\Gamma\right)}$-invariant.
\end{lem}

The second lemma is a refinement of the well-known fact that an ergodic action is of type $\mathrm{III}_1$ if, and only if, its Maharam extension is ergodic.

\begin{lem}
\label{Lemma: Renormalization 2}
Let $\Gamma\curvearrowright\left(X,\mu\right)$ be a countable group of automorphisms on a standard Borel probability space. Then $\mathrm{e}\left(\Gamma\right)=\mathbb{R}$ if, and only if, every $\widetilde{\Gamma}$-invariant function
$F\in L^1\left(\widetilde{X},\widetilde{\mu}\right)$ is of the form $F\left(x,t\right)=f\left(x\right)$ for $\widetilde{\mu}$-a.e. $\left(x,t\right)\in\widetilde{X}$.\\
Of course, when $\Gamma$ is ergodic then the only such functions are the constant functions.
\end{lem}

\begin{proof}[Proof of Theorem \ref{Theorem: Hopf Argument} assuming Lemmas \ref{Lemma: Renormalization 1} and \ref{Lemma: Renormalization 2}]\
By Lemma \ref{Lemma: Renormalization 2} we need to show that if every $\widetilde{\Gamma}$-invariant function $F:\widetilde{X}\to\mathbb{R}$ is of the form $F\left(x,t\right)=f\left(x\right)$ in $L^1\left(\widetilde{\mu}\right)$, then the same holds for every $\widetilde{T_G}$-invariant function. This is straightforward from Lemma \ref{Lemma: Renormalization 1}.
\end{proof}

\begin{proof}[Proof of Lemma \ref{Lemma: Renormalization 1}]\ The main idea of the proof is similar to that of \cite[Lemma~3.1]{bjorklund2020ergodicity}. Denote by $\widetilde{\mathcal{I}}$ the sigma-algebra of Borel $\widetilde{T}_{G}$-invariant sets. We show that for every $V\in\mathscr{R}\left(T_G;\Gamma\right)$ there exists a positive function $v\left(x\right)$, such that for every $F\in L^1\left(\widetilde{X},\widetilde{\mu}\right)$ we have that
\begin{equation}
\label{eq:6}
v\left(x\right)^{-1}\mathbf{E}\left(F\left(x,t\right)\mid\widetilde{\mathcal{I}}\right)\leq\mathbf{E}\left(F\mid\widetilde{\mathcal{I}}\right)\circ\widetilde{V}\left(x,t\right)\leq v\left(x\right)\mathbf{E}\left(F\left(x,t\right)\mid\widetilde{\mathcal{I}}\right)
\end{equation}
for $\widetilde{\mu}$-a.e. $\left(x,t\right)\in\widetilde{X}$. It will follow that if $A\in\widetilde{\mathcal{I}}$ so that the function $F\left(x,t\right)=\mathbf{1}_{A}\left(x,t\right)$ is $\widetilde{T}_{G}$-invariant, then for every $V\in\mathscr{R}\left(T_{G};\Gamma\right)$ we have
$$v\left(x\right)^{-1}F\left(x,t\right)\leq F\left(\widetilde{V}\left(x,t\right)\right)\leq v\left(x\right)F\left(x,t\right)$$
for $\widetilde{\mu}$-a.e. $\left(x,t\right)\in\widetilde{X}$. Since $F$ is taking only the values 0 and 1 it will follow that $F\left(\widetilde{V}\left(x,t\right)\right)=F\left(x,t\right)$ so that $A$ is $\widetilde{V}$-invariant. Now every $\widetilde{T}_G$-invariant $L^1\left(\widetilde{X},\widetilde{\mu}\right)$-function is an $L^1\left(\widetilde{X},\widetilde{\mu}\right)$-limit and hence $\widetilde{\mu}$-a.e. limit of of a sequence of $\widetilde{T}_G$-invariant simple functions, so proving \eqref{eq:6} will finish the proof of Lemma \ref{Lemma: Renormalization 1}.

We prove \eqref{eq:6}. Instead of the usual infinite measure $\widetilde{\mu}$ on $\widetilde{X}$ we consider the equivalent probability measure $\widehat{\mu}$ on $\widetilde{X}$ defined by
$$d\widehat{\mu}\left(x,t\right)=d\mu\left(x\right)e\left(t\right)dt,\text{ where }e\left(t\right):=\exp\left(-\left|t\right|\right)/2.$$
Then the Maharam extension $\widetilde{T}_{G}$ is nonsingular and conservative with respect to $\widehat{\mu}$ as well. We also have that $L^1\left(\widetilde{X},\widetilde{\mu}\right)\subset L^1\left(\widetilde{X},\widehat{\mu}\right)$, so it is enough to prove \eqref{eq:6} for $L^1\left(\widetilde{X},\widehat{\mu}\right)$. Consider the class $\mathcal{L}\subset L^{1}\left(\widetilde{X},\widehat{\mu}\right)$ of functions of the form $L\left(x,t\right)=\phi\left(x\right)\varphi\left(t\right)$, where $\phi\in L^{1}\left(X,\mu\right)$ and $\varphi\in L^{1}\left(\mathbb{R},e\left(t\right)dt\right)$ are both uniformly continuous bounded function satisfying $\inf_{x\in X}\phi\left(x\right)>0$ and $\inf_{t\in\mathbb{R}}\varphi\left(t\right)>0$. Then the linear space generated by $\mathcal{L}$ is dense in $L^{1}\left(\widetilde{X},\widehat{\mu}\right)$. Passing to subsequences converging $\widehat{\mu}$-a.e. and using the continuity of the conditional expectation, it is enough to establish \eqref{eq:6} for the functions of $\mathcal{L}$.

Given any $L\left(x,t\right)=\phi\left(x\right)\varphi\left(t\right)\in\mathcal{L}$, by the ratio ergodic theorem of \cite[Theorem~0.4]{danilenko2019weak} for the Maharam extension $\widetilde{T}_{G}$, there exists for $L$ an increasing sequence of finite sets $G_{1}\subset G_{2}\subset\dotsc$ whose union is $G$, such that
\begin{equation}
\mathbf{E}\left(L\left(x,t\right)\mid\widetilde{\mathcal{I}}\right)=\lim_{N\to\infty}\frac{\sum_{g\in G_{N}}\widetilde{T}_{g}'\left(x,t\right)L\left(\widetilde{T}_{g}\left(x,t\right)\right)}{\sum_{g\in G_{N}}\widetilde{T}_{g}'\left(x,t\right)}
\end{equation}
and
\begin{equation}
\label{eq:7}
\mathbf{E}\left(L\mid\widetilde{\mathcal{I}}\right)\circ\widetilde{V}\left(x,t\right)=\lim_{N\to\infty}\frac{\sum_{g\in G_{N}}\widetilde{T}_{g}'\left(\widetilde{V}\left(x,t\right)\right)L\left(\widetilde{T}_{g}\left(\widetilde{V}\left(x,t\right)\right)\right)}{\sum_{g\in G_{N}}\widetilde{T}_{g}'\left(\widetilde{V}\left(x,t\right)\right)}
\end{equation}
for $\widetilde{\mu}$-a.e. $\left(x,t\right)\in\widetilde{X}$, where here and in the rest of this proof the notation $\widetilde{T}_{g}'$ refers to the Radon-Nikodym derivative with respect to $\widehat{\mu}$.

We first claim that
\begin{equation}
\label{eq:8}
\frac{L\left(\widetilde{T}_{G}\left(\widetilde{V}\left(x,t\right)\right)\right)}{L\left(\widetilde{T}_{G}\left(x,t\right)\right)}=\frac{\phi\left(T_{g}\left(Vx\right)\right)}{\phi\left(T_{g}\left(x\right)\right)}\frac{\varphi\left(t-\log\left(T_{g}\circ V\right)'\left(x\right)\right)}{\varphi\left(t-\log T_{g}'\left(x\right)\right)}\xrightarrow[g\to\infty]{}1
\end{equation}
for $\widetilde{\mu}$-a.e. $\left(x,t\right)\in\widetilde{X}$. The first factor converges to 1 as $g\to\infty$ for $\mu$-a.e. $x\in X$, since $\left(x,Vx\right)$ is an asymptotic pair and by the choice of $\phi$. The second factor also converges to 1 as $g\to\infty$ for $\widehat{\mu}$-a.e. $\left(x,t\right)\in\widetilde{X}$, since $V\in\mathscr{R}\left(T_{G};\Gamma\right)$ and by the choice of $\varphi$.

We also claim that there are positive functions $v_{0}\left(x\right)$ and $v_{1}\left(x\right)$ depending only on $V$, such that
\begin{equation}
\label{eq:9}
v_{0}\left(x\right)\leq\liminf_{g\to\infty}\frac{\widetilde{T}_{g}'\left(\widetilde{V}\left(x,t\right)\right)}{\widetilde{T}_{g}'\left(x,t\right)}\leq\limsup_{g\to\infty}\frac{\widetilde{T}_{g}'\left(\widetilde{V}\left(x,t\right)\right)}{\widetilde{T}_{g}'\left(x,t\right)}\leq v_{1}\left(x\right).
\end{equation}
for $\widetilde{\mu}$-a.e. $\left(x,t\right)\in\widetilde{X}$. The reason for that is the following. It is straightforward to verify that the Radon-Nikodym derivatives with respect to $\widehat{\mu}$ take the form
$$\frac{\widetilde{T}_{g}'\left(\widetilde{V}\left(x,t\right)\right)}{\widetilde{T}_{g}'\left(x,t\right)}=\frac{T_{g}'\left(Vx\right)}{T_{g}'\left(x\right)}\frac{e\left(t-\log T_{g}'\left(Vx\right)\right)}{e\left(t-\log T_{g}'\left(x\right)\right)}\text{ for }\widehat{\mu}\text{-a.e. }\left(x,t\right)\in\widetilde{X}.$$
Since $V\in\mathscr{R}\left(T_{G};\Gamma\right)$ we have that
$$\frac{T_{g}'\left(Vx\right)}{T_{g}'\left(x\right)}\xrightarrow[g\to\infty]{}\left(V'\left(x\right)\right)^{-1}\text{ for }\mu\text{-a.e. }x\in X.$$
Using the bound $\exp\left(\left|s\right|\right)^{-1}\leq e\left(t+s\right)/e\left(t\right)\leq\exp\left(\left|s\right|\right)$ we have that
$$\frac{e\left(t-\log T_{g}'\left(Vx\right)\right)}{e\left(t-\log T_{g}'\left(x\right)\right)}\leq\exp\left(\left|\log\frac{T_{g}'\left(Vx\right)}{T_{g}'\left(x\right)}\right|\right)\xrightarrow[g\to\infty]{}\exp\left(\left|\log V'\left(x\right)\right|\right),$$
and that
$$\frac{e\left(t-\log T_{g}'\left(Vx\right)\right)}{e\left(t-\log T_{g}'\left(x\right)\right)}\geq\exp\left(\left|\log\frac{T_{g}'\left(Vx\right)}{T_{g}'\left(x\right)}\right|\right)^{-1}\xrightarrow[g\to\infty]{}\exp\left(\left|\log\left(V'\left(x\right)\right)\right|\right)^{-1}.$$
Then \eqref{eq:9} holds for
$$v_{0}\left(x\right):=\left(V'\left(x\right)\right)^{-1}\exp\left(\left|\log V'\left(x\right)\right|\right)^{-1}\text{ and }v_{1}\left(x\right):=\left(V'\left(x\right)\right)^{-1}\exp\left(\left|\log V'\left(x\right)\right|\right).$$
Recall that by the conservativeness of $\widetilde{T}_{G}$ we have
$$\lim_{N\to\infty}\sum_{g\in G_{N}}\widetilde{T}_{g}'\left(x,t\right)=\infty\text{ for }\widehat{\mu}\text{-a.e. }\left(x,t\right)\in\widetilde{X},$$
so that plugging \eqref{eq:8} and \eqref{eq:9} into \eqref{eq:7}, we obtain that \eqref{eq:6} holds for the positive function $v\left(x\right)=v_{1}\left(x\right)/v_{0}\left(x\right)$.
\end{proof}

Before we prove Lemma \ref{Lemma: Renormalization 2}, let us mention some basic facts about the ergodic decomposition and its relation to the Maharam extension. Let $\Gamma\curvearrowright\left(X,\mu\right)$ be a countable group of automorphisms. As described by Bowen following Zimmer \cite{bowen2014}, the ergodic decomposition of this action is the standard Borel space $\left(\mathscr{E},\nu\right)$, where $\mathscr{E}$ is the space of all $\Gamma$-nonsingular $\Gamma$-ergodic probability measures on $X$ and $\nu$ is a Borel measure on $\mathscr{E}$ that satisfies
$$\mu\left(\cdot\right)=\intop_{\mathscr{E}}\kappa\left(\cdot\right)d\nu\left(\kappa\right).$$
Moreover, there exists a regular choice of Radon--Nikodym cocycle with respect to the ergodic decomposition in the following sense. There exists a Borel function $\varphi_{\gamma}\left(x\right):\Gamma\times X\to\mathbb{R}$ that satisfies the cocycle identity
\begin{equation}
\label{eq:2}
\varphi_{\gamma_{1}}\left(x\right)+\varphi_{\gamma_{2}}\left(\gamma_{1}x\right)=\varphi_{\gamma_{1}\gamma_{2}}\left(x\right),\quad\forall \gamma_{1},\gamma_{2}\in\Gamma,\forall x\in X,
\end{equation}
such that
\begin{equation}
\label{eq:3}
\varphi_{\gamma}\left(x\right)=\log\frac{d\kappa\circ \gamma}{d\kappa}\left(x\right),\quad\forall \gamma\in\Gamma,\,\nu\text{-a.e. }\kappa\in\mathscr{E},\,\kappa\text{-a.e. }x\in X.
\end{equation}
Fix such a cocycle and a corresponding $\nu$-full measure set $\mathscr{E}_{0}\subset\mathscr{E}$. Consider the Maharam extension $\widetilde{\Gamma}\curvearrowright\left(\widetilde{X},\widetilde{\mu}\right)$ and the Maharam extensions $\widetilde{\Gamma}\curvearrowright\left(\widetilde{X},\widetilde{\kappa}\right)$ for the ergodic components $\kappa\in \mathscr{E}_0$. Note that the property
\begin{equation}
\label{eq:10}
\widetilde{\mu}\left(\cdot\right)=\intop_{\mathscr{E}_{0}}\widetilde{\kappa}\left(\cdot\right)d\nu\left(\kappa\right)
\end{equation}
can be easily verified on sets in $\mathcal{B}\left(X\right)\times\mathcal{B}\left(\mathbb{R}\right)$ using the Fubini theorem, thus it holds for all Borel sets of $\widetilde{X}$.

\begin{proof}[Proof of Lemma \ref{Lemma: Renormalization 2}]\
We prove the Lemma for indicator functions, showing that $\mathrm{e}\left(\Gamma\right)=\mathbb{R}$ if, and only if, every $\widetilde{\Gamma}\curvearrowright\left(\widetilde{X},\widetilde{\mu}\right)$-invariant set $E\subset\widetilde{X}$ with $\widetilde{\mu}\left(E\right)>0$ is of the form $E=E'\times\mathbb{R}\mod\widetilde{\mu}$ for some $E'\subset X$.

One implication is standard: Suppose that $r\in\mathbb{R}\backslash\mathrm{e}\left(\Gamma,\mu\right)$. Then there exists $\epsilon>0$ and some $E_0\subset X$ with $\mu\left(E_0\right)>0$, such that $-\log\gamma'\left(x\right)\notin\left(r,r+\epsilon\right)$ for $\mu$-a.e. $x\in E_0$ for every $\gamma\in\Gamma$. Consider the $\widetilde{\Gamma}\curvearrowright\left(\widetilde{X},\widetilde{\mu}\right)$-invariant set
$$E:=\bigcup_{\gamma\in\Gamma}\widetilde{\gamma}\left(E_0\times\left(0,\epsilon/2\right)\right)\subset\widetilde{X}.$$
Note that for $\widetilde{\mu}$-a.e. $\left(x,t\right)\in E_0\times\left(0,\epsilon/2\right)$ and every $\gamma\in\Gamma$,
$$\mathrm{proj}_{\mathbb{R}}\left(\widetilde{\gamma}\left(x,t\right)\right)=t-\log\gamma'\left(x\right)\notin\left(r+\epsilon/2,r+\epsilon\right).$$
Hence $\mathrm{proj}_{\mathbb{R}}\left(E\right)\cap\left(r+\epsilon/2,r+\epsilon\right)=\emptyset$ so that $E$ can not be of the form $E=E'\times\mathbb{R}$.

For the other implication, assume that $\mathrm{e}\left(\Gamma,\mu\right)=\mathbb{R}$ and let $E\subset\widetilde{X}$ be a $\widetilde{\Gamma}\curvearrowright\left(\widetilde{X},\widetilde{\mu}\right)$-invariant Borel set. For $x\in X$ let $E_{x}=\left\{ t\in\mathbb{R}:\left(x,t\right)\in E\right\}$. Let $\lambda$ be the Lebesgue measure on $\mathbb{R}$. Fix a cocycle $\varphi$ satisfying \eqref{eq:2} and \eqref{eq:3} and a corresponding $\nu$-full measure set of ergodic components $\mathscr{E}_0\subset\mathscr{E}$. Note that since $E$ is $\widetilde{\Gamma}\curvearrowright\left(\widetilde{X},\widetilde{\mu}\right)$-invariant, using formula \eqref{eq:10} we have that
$$0=\widetilde{\mu}\left(E\triangle\widetilde{\gamma}^{-1}E\right)=\intop_{\mathscr{E}_0}\widetilde{\kappa}\left(E\triangle\widetilde{\gamma}^{-1}E\right)d\nu\left(\kappa\right),\quad \gamma\in\Gamma,$$
so for every $\gamma\in\Gamma$ there is a $\nu$-full measure set $\mathscr{E}_{\gamma}\subset\mathscr{E}_{0}$ such that $\widetilde{\kappa}\left(E\triangle\widetilde{\gamma}^{-1}E\right)=0$ for every $\kappa\in\mathscr{E}_{\gamma}$. Letting the $\nu$-full measure set $\mathscr{E}_{1}:=\bigcap_{\gamma\in\Gamma}\mathscr{E}_{\gamma}\subset\mathscr{E}_0$, we get that $E$ is $\widetilde{\Gamma}\curvearrowright\left(\widetilde{X},\widetilde{\kappa}\right)$-invariant for every $\kappa\in\mathscr{E}_1$.

By Bowen's theorem \cite[Theorem 2.1]{bowen2014} and our assumption, the ratio set of $\nu$-a.e. $\kappa\in\mathscr{E}$ satisfies $\mathrm{e}\left(\Gamma,\kappa\right)=\mathrm{e}\left(\Gamma,\mu\right)=\mathbb{R}$. Let $\mathscr{E}_2\subset \mathscr{E}_1$ be a $\nu$-full measure set satisfying this property. By Schmidt's theorem \cite[Theorem 5.2]{schmidt1977} \cite[Theorem 8.2.4]{aaronson1997}, the ratio set of an ergodic action is the same as its \textit{periods set}, which means that in our case for every $\kappa\in\mathscr{E}_2$,
$$\mathbb{R}=\mathrm{e}\left(\Gamma,\kappa\right)=\left\{ r\in\mathbb{R}:S_{r}F=F\,\text{ mod }\widetilde{\kappa},\,\forall \widetilde{\Gamma}\curvearrowright\left(\widetilde{X},\widetilde{\kappa}\right)\text{-invariant set }F\subset\widetilde{X}\right\},$$
where for $r\in\mathbb{R}$, $S_r\left(x,t\right)=\left(x,t+r\right)$. Then for every $\kappa\in\mathscr{E}_{2}$ we have that $S_rE=E\mod\widetilde{\kappa}$ for every $r\in\mathbb{R}$, hence $E_{x}=E_{x}-r\mod\lambda$ for every $r\in\mathbb{R}$. As the Lebesgue measure $\lambda$ is translation-invariant, it has the property that every pair of positive Lebesgue measure sets $A_{0}$ and $A_{1}$ admits a positive length interval $I$ such that $\lambda\left(A_{0}\cap\left(A_{1}-r\right)\right)>0$ for every $r\in I$, so it is impossible that both $\lambda\left(E_{x}\right)>0$ and $\lambda\left(\mathbb{R}\backslash E_{x}\right)>0$. That is, $E_{x}\in\left\{ \emptyset,\mathbb{R}\right\} \mod\lambda$ for $\kappa$-a.e. $x\in X$ for every $\kappa\in\mathscr{E}_2$ which is a $\nu$-full measure set. By the Fubini theorem and the above ergodic decomposition, this is equivalent to $E=E'\times\mathbb{R}\mod\widetilde{\mu}$ where $E'$ is the set of all $x\in X$ for which $E_x=\mathbb{R}\mod\lambda$, so the proof is complete.
\end{proof}

 \section{Markov Subshift of Finite Type (MSFT)}

A \textit{Markov Subshift of Finite Type} (Markov SFT or MSFT) is a measure space $\left(X_A,\mu\right)$ where $X_A$ is an SFT on a finite state space $\mathcal{S}$ with adjacency matrix $A$, and $\mu$ is a Markov measure on $X_A$ which is compatible with $A$ in the sense that if the transition matrices of $\mu$ are $\left(P_{n}:n\in\mathbb{Z}\right)$ then
$$P_{n}\left(s,t\right)>0\iff A\left(s,t\right)=1,\quad s,t\in\mathcal{S}, n\in\mathbb{Z}.$$
For a pair of integers $n<m$ denote
$$P^{\left(n,m\right)}=P_n\cdot\dotsm\cdot P_m,$$
which is a row-stochastic $\left|\mathcal{S}\right|\times\left|\mathcal{S}\right|$-matrix that has the interpretation
$$P^{\left(n,m\right)}\left(s,t\right)=\mu\left(X_{m+1}=t\mid X_n=s\right),\quad s,t\in\mathcal{S}.$$
Note that if $\left(X_A,\mu\right)$ is a topologically-mixing MSFT with $A^M>0$ then
$$P^{\left(n,n+M-1\right)}\left(s,t\right)>0,\quad n\in\mathbb{Z}, s,t\in\mathcal{S}.$$

The following proposition will be used constantly in our work and we include its proof in Appendix \ref{Appendix: Mixing Properties of Markov Measures}.

\begin{prop}
\label{Proposition: Mixing in MSFT}
Let $\left(X_A,\mu\right)$ be a topologically-mixing MSFT with $A^M>0$, that satisfies the Doeblin condition \eqref{Doeblin} for $\delta>0$. Let $\left(P_n:n\in\mathbb{Z}\right)$ be the transition matrices of $\mu$ and let $\left(\pi_n:n\in\mathbb{Z}\right)$ be the coordinates distributions of $\mu$. Then the following properties hold.
\begin{enumerate}
	\item For every $n\in\mathbb{Z}$,
	$$\delta^{M}\leq\pi_{n}\left(s\right)\leq1-\delta^{M},\quad s\in\mathcal{S}.$$
	\item For every integer $N\geq M$ and $n\in\mathbb{Z}$,
	$$\delta^{M}\leq P^{\left(n,n+N\right)}\left(s,t\right)\leq1-\delta^{M},\quad s,t\in\mathcal{S}.$$
	\item There exists a constant $C\left(\delta,M\right)\in\left(0,1\right)$ depending only on $\delta$ and $M$, such that for every $n,m\in\mathbb{Z}$ and every pair of Borel sets $E\in\sigma\left(\dotsc,X_{n-1},X_n\right)$ and $F\in\sigma\left(X_m,X_{m+1},\dotsc\right)$, if $m-n\geq M$ then
$$C\left(\delta,M\right)\mu\left(E\right)\mu\left(F\right)\leq\mu\left(E\cap F\right)\leq C\left(\delta,M\right)^{-1}\mu\left(E\right)\mu\left(F\right).$$
\end{enumerate}
\end{prop}

Here we establish a deterministic criteria for equivalence of Markov measures which will be fundamental to our work. The Hahn-Lebesgue decomposition of one Markov measure with respect to another is known but unlike the Kakutani dichotomy in product measures, in general it is not a 0-1 event and there is no a deterministic criteria to distinguish between the alternatives. For one-sided Markov chains some authors assumed tail triviality as well as other regularity assumptions to establish such a deterministic criteria; see for instance \cite{lodkin1971absolute, lepage1975likelihood, danilenko2019}. Here we take a different approach, which under the assumption of the Doeblin condition makes no reference to tail triviality. We provide the proof of Theorem \ref{Theorem: Kakutani Criteria} as well as a detailed background in Appendix \ref{Appendix: Criteria for Equivalence of MM}.

\begin{thm}
\label{Theorem: Kakutani Criteria}
Let $\nu=\nu_{\left(P_n:n\in\mathbb{Z}\right)}$ and $\mu=\mu_{\left(Q_n:n\in\mathbb{Z}\right)}$ be Markov measures on a topologically-mixing SFT $X_A$, both satisfy the Doeblin condition \ref{Doeblin}. Then $\nu\ll\mu$ if, and only if,
$$\sum_{n\geq1}\sum_{s,u,v,t\in\mathcal{S}}\mathrm{d}_{n}^{2}\left[\nu,\mu\right]\left(s,u,v,t\right)<\infty,$$
where for $n\geq1$ and $s,t,u,v\in\mathcal{S}$ we denote the numbers
$$\mathrm{d}_{n}^{2}\left[\nu,\mu\right]\left(s,u,v,t\right):=\left(\sqrt{\widehat{P}_{-n}\left(u,s\right)P_{n}\left(v,t\right)}-\sqrt{\widehat{Q}_{-n}\left(u,s\right)Q_{n}\left(v,t\right)}\right)^{2}.$$
In particular, since $\mathrm{d}_{n}^{2}\left[\nu,\mu\right]\left(s,u,v,t\right)=\mathrm{d}_{n}^{2}\left[\mu,\nu\right]\left(s,u,v,t\right)$ for all $n\geq1$ and $s,u,v,t\in\mathcal{S}$, it follows that $\nu\ll\mu\iff\mu\ll\nu$ so that every such two measures are either equivalent or that none of them is absolutely continuous with respect to the other.
\end{thm}

\begin{cor}
\label{Corollary: Nonsingularity of the Shift}
Let $\left(X_A,\mu\right)$ be a topologically-mixing MSFT that satisfies the Doeblin condition. The coefficients for the nonsingularity of the shift $T$ are
$$\mathrm{d}_{n}^{2}\left[\mu,\mu\circ T^{-1}\right]\left(s,u,v,t\right)=\left(\sqrt{\widehat{P}_{-n}\left(u,s\right)P_{n}\left(v,t\right)}-\sqrt{\widehat{P}_{-\left(n+1\right)}\left(u,s\right)P_{n+1}\left(v,t\right)}\right)^{2}.$$
Thus, when the shift is nonsingular, using the stochasticity of the matrices $P_n$ and $\widehat{P}_n$ for all $n\in\mathbb{Z}$ we get by Theorem \ref{Theorem: Kakutani Criteria} that
$$P_{n-1}\left(v,t\right)-P_{n}\left(v,t\right)\xrightarrow[n\to\infty]{}0,\quad v,t\in\mathcal{S}$$
and
$$\widehat{P}_{-n}\left(u,s\right)-\widehat{P}_{-\left(n+1\right)}\left(u,s\right)\xrightarrow[n\to\infty]{}0,\quad s,u\in\mathcal{S}.\qedhere$$
\end{cor}

\begin{cor}
\label{Corollary: Equivalence to Homogeneous}
Let $\left(X_A,\mu\right)$ be a MSFT and let $\left(P_n:n\in\mathbb{Z}\right)$ be the sequence of transition matrices of $\mu$. Then a necessary condition for $\mu$ to be equivalent to a homogeneous Markov measure $\nu$ defined by a matrix $Q$, is that
$$\lim_{\left|n\right|\to\infty}P_{n}=Q.$$
If this holds, then $\mu$ is equivalent to $\nu$ if, and only if,
$$\sum_{n\geq1}\sum_{s,u,v,t\in\mathcal{S}}\left(\sqrt{\widehat{P}_{-n}\left(u,s\right)P_{n}\left(v,t\right)}-\sqrt{\widehat{Q}\left(u,s\right)Q\left(v,t\right)}\right)^2<\infty.$$
\end{cor}

\subsection{Renormalization in MSFT}

Let $\left(X_A,\mu\right)$ be a topologically-mixing MSFT that satisfies the Doeblin condition. Consider the action of $G=\mathbb{Z}$ by the shift $T:\left(X_A,\mu\right)\to\left(X_A,\mu\right)$. Let $\Pi$ be the group of all permutations of $\mathbb{Z}$ that change only finitely many coordinates, and consider the equivalence relation that consists of all $\left(x,y\right)\in X_A\times X_A$ for which $y=\pi x$ for some $\pi\in\Pi$. This is a Borel countable equivalence relation, so by the Feldman--Moore Theorem it is the orbital equivalence relation of a countable group $\Pi_{A}$ of nonsingular automorphisms of $X_A$. The renormalization full-group $\mathscr{R}\left(T;\Pi_A\right)$ is usually a proper subgroup of $\left[\Pi_A\right]$. We write $\mathscr{R}_A$ and $\left[\left[\mathscr{R}_A\right]\right]$ for the renormalization full-group $\mathscr{R}\left(T;\Pi_A\right)$ and for the renormalization pseudo full-group $\mathscr{R}\left(T;\left[\left[\Pi_A\right]\right]\right)$, respectively.

Here we identify a collection of elements of $\left[\Pi_A\right]$ inside $\mathscr{R}\left(T;\Pi_A\right)$.

A \textit{block} $B$ in an SFT $X_{A}$ is a finite sequence $B=\left[b_1,\dotsc,b_L\right]$ of symbols from $\mathcal{S}$, such that $A\left(b_l,b_{l+1}\right)=1$ for $1\leq l\leq L-1$. For such $B$ we write $L=\mathrm{Length}\left(B\right)$. For a block $B$ in $X_{A}$ with $L=\mathrm{Length}\left(B\right)$ and for $i\in\mathbb{Z}$, we have the corresponding cylinder
$$B\left(i\right):=\left\{x\in X_A:\left[x_{i},\dotsc,x_{i+L-1}\right]=B\right\} \subset X_{A}.$$
A pair $\left(B,B'\right)$ of two blocks in $X_A$ is called an \textit{admissible pair} in $X_{A}$ with length $L\geq1$, and we write $\mathrm{Length}\left(B,B'\right)=L$, if it satisfies the following properties.
\begin{enumerate}
	\item $\mathrm{Length}\left(B\right)=\mathrm{Length}\left(B'\right)=L$.
	\item $B$ and $B'$ have the same first symbol.
	\item $B$ and $B'$ have the same last symbol.
\end{enumerate}
To avoid trivialities we always assume that $B\neq B'$. In particular we always have $\mathrm{Length}\left(B,B'\right)\geq3$. An example for admissible pair $\left(B,B'\right)$ in the Golden Mean SFT is the one of length $L=4$ defined by $B=\left[0,0,0,0\right]$ and $B'=\left[0,2,1,0\right]$.

\begin{dfn}
\label{Definition: Admissible Configuration}
Let $X_A$ be a topologically-mixing SFT with $A^M>0$ for some $M\geq1$. An \emph{admissible configuration} in $X_A$ is a sequence $\left(B_{k}\left(i_{k}\right),B_{k}'\left(j_{k}\right)\right)$, $k\geq1$, built out of the following ingredients.
\begin{itemize}
	\item A sequence $\left(B_{k},B_{k}'\right)$, $k\geq1$, of admissible pairs in $X_A$ with some $L\geq1$ such that
	$$\mathrm{Length}\left(B_k,B_k'\right)\leq L\quad\text{for all }k\geq1.$$
	\item A sequence $\left(j_{k}:k\geq1\right)$ of positive integers with
	$$j_{k+1}-j_{k}\geq L+M,\quad k\geq1.$$
	\item A sequence $\left(i_{k}:k\geq1\right)$ of negative integers with
	$$i_{k}-i_{k+1}\geq L+M,\quad k\geq1.$$
\end{itemize}
\end{dfn}

\begin{dfn}
Let $\mu$ be a Markov measure on $X_A$ defined by $\left(P_n:n\in\mathbb{Z}\right)$. Let $\left(B,B'\right)$ be an admissible pair in $X_A$ and $i,j\in\mathbb{Z}$ with $\left|i-j\right|\geq\mathrm{Length}\left(B,B'\right)$. Denote 
$$E_{i,j}:=B\left(i\right)\cap B'\left(j\right)\text{ and }E'_{i,j}:=B'\left(i\right)\cap B\left(j\right).$$
We define two types of elements of the pseudo full-group $\left[\left[\Pi_{A}\right]\right]$.
\begin{itemize}
	\item The corresponding \emph{ asymmetric admissible permutation} is of the form
	$$V:E_{i,j}\to E'_{i,j}$$
	and is defined to exchange the block $B$ in the coordinates $\left\{i,\dotsc,i+L-1\right\}$ with the block $B'$ in the coordinates $\left\{j,\dotsc,j+L-1\right\}$. We write such element by
$$V:B\left(i\right)\rightleftharpoons B'\left(j\right)\in\left[\left[\Pi_{A}\right]\right].$$
	\item The corresponding \emph{symmetric admissible permutation} is of the form
	$$V:E_{i,j}\cup E'_{i,j}\to E_{i,j}\cup E'_{i,j}$$
	and is defined on $E_{i,j}$ by $V:B\left(i\right)\rightleftharpoons B'\left(j\right)$ and on $E'_{i,j}$ by $V:B'\left(i\right)\rightleftharpoons B\left(j\right)$. We write such element by
$$V:B\left(i\right)\circlearrowleft B'\left(j\right)\in\left[\left[\Pi_{A}\right]\right].$$
\end{itemize}
\end{dfn}

Note that if we define admissible permutations to be the identity mappings outside of their domains, then symmetric admissible permutations remain one-to-one while asymmetric admissible permutations are no longer one-to-one.

Let us establish a notation. Given a Markov measure $\mu$ on $X_A$ defined by $\left(P_n:n\in\mathbb{Z}\right)$, for a block $B=\left[b_1,b_2,\dotsc,b_L\right]$ and $i\in\mathbb{Z}$ we write
$$P_i\left(B\right):=P_i\left(b_1,b_2\right)\dotsm P_{i+L-1}\left(b_{L-1},b_L\right).$$
The following formula is a direct computation using the properties of admissible permutations and of the Radon--Nikodym derivative.

\begin{claim}
\label{Claim: Admissible Permutation Derivative}
For an asymmetric admissible permutation $V:B\left(i\right)\rightleftharpoons B'\left(j\right)$ we have
$$V'\left(x\right)=\frac{P_{i}\left(B'\right)P_{j}\left(B\right)}{P_{i}\left(B\right)P_{j}\left(B'\right)}\quad\text{for }x\in B\left(i\right)\cap B'\left(j\right).$$
In particular, $V'$ is taking exactly one value on $B\left(i\right)\cap B'\left(j\right)$ and this value depends only on the coordinates of its admissible pair.\\
Similarly, for a symmetric admissible permutation $V:B\left(i\right)\circlearrowleft B'\left(j\right)$ we have
$$V'\left(x\right)=\begin{cases}
{\displaystyle \frac{P_{i}\left(B'\right)P_{j}\left(B\right)}{P_{i}\left(B\right)P_{j}\left(B'\right)}} & \text{for }x\in B\left(i\right)\cap B'\left(j\right)\\
\\
{\displaystyle \frac{P_{i}\left(B\right)P_{j}\left(B'\right)}{P_{i}\left(B'\right)P_{j}\left(B\right)}} & \text{for }x\in B'\left(i\right)\cap B\left(j\right)
\end{cases}.$$
\end{claim}

For an admissible configuration $\left(B_{k}\left(i_{k}\right),B_{k}'\left(j_{k}\right)\right)$, $k\geq1$, we denote by $\left(D_k:k\geq1\right)$ the sequence of numbers
\begin{align}
\label{eq:15}
D_{k}:=\log\left(\frac{P_{i_{k}}\left(B_{k}'\right)P_{j_{k}}\left(B_{k}\right)}{P_{i_{k}}\left(B_{k}\right)P_{j_{k}}\left(B_{k}'\right)}\right)\asymp P_{i_{k}}\left(B_{k}'\right)P_{j_{k}}\left(B_{k}\right)-P_{i_{k}}\left(B_{k}\right)P_{j_{k}}\left(B_{k}'\right),
\end{align}
where the approximation is by the approximation of the logarithm in \ref{Fact: Log Approximation}.

\begin{lem}
\label{LemmaAdmissiblePermutations}
Let $\left(X_A,\mu\right)$ be a topologically-mixing MSFT that satisfies the Doeblin condition \ref{Doeblin} and suppose that the shift is nonsingular with respect to $\mu$. Taking the metric on $X_A$ as in Example \ref{Example:BernoulliShift}, we have that every symmetric admissible permutation belongs to $\mathscr{\mathscr{R}}_A$. Similarly, every asymmetric admissible permutation belongs to $\left[\left[\mathscr{R}_A\right]\right]$.
\end{lem}

\begin{proof}
Since any symmetric admissible permutation is defined by two asymmetric admissible permutations on disjoint domains, it is enough to consider only the asymmetric case. Let $V:B\left(i\right)\rightleftharpoons B'\left(j\right)$, when $B=\left[b_1,b_2\dotsc,b_L\right]$ and $B'=\left[b_1',b_2',\dotsc,b_L'\right]$ with $b_1=b_1'$ and $b_L=b_L'$. Note that for every $n\in\mathbb{Z}$ and $x\in B\left(i\right)\cap B'\left(j\right)$, $T^n\left(x\right)$ and $T^n\left(Vx\right)$ differ only in the coordinates $\left\{j-n,\dotsc,j-n+L-1\right\}$ and $\left\{i-n,\dotsc,i-n+L-1\right\}$, then since $b_1=b_1'$ and $b_L=b_L'$ we have
$$\frac{\left(T^{n}\circ V\right)'\left(x\right)}{\left(T^{n}\right)'\left(x\right)}=\frac{d\mu\circ T^{n}\circ V}{d\mu\circ T^{n}}\left(x\right)=\prod_{l=1}^{L}\frac{P_{j-n+l-1}\left(b_{l},b_{l+1}\right)}{P_{j-n+l-1}\left(b_{l}',b_{l+1}'\right)}\cdot\frac{P_{i-n+l-1}\left(b_{l}',b_{l+1}'\right)}{P_{i-n+l-1}\left(b_{l},b_{l+1}\right)}.$$
By the Doeblin condition and Corollary \ref{Corollary: Nonsingularity of the Shift}, for every $s,t\in\mathcal{S}$ with $A\left(s,t\right)=1$ and every $1\leq l\leq L$,
$$\frac{P_{j-n+l-1}\left(s,t\right)}{P_{i-n+l-1}\left(s,t\right)}\asymp 1+\left(P_{j-n+l-1}\left(s,t\right)-P_{i-n+l-1}\left(s,t\right)\right)\xrightarrow[\left|n\right|\to\infty]{}1.$$
As the length of the product is bounded by $L$ uniformly in $n$, this shows that $\frac{\left(T^{n}\circ V\right)'\left(x\right)}{\left(T^{n}\right)'\left(x\right)}\xrightarrow[\left|n\right|\to\infty]{}1$ for $\mu$-a.e. $x\in X_A$.
\end{proof}

\section{Proof of the Divergent Scenario}

Here we prove Theorem \ref{Theorem: Theorem B}.

\begin{lem}
\label{Lemma: Partial Limits Essential Values}
For every admissible configuration $\left(B\left(i_k\right),B'\left(j_k\right)\right)$, $k\geq1$, the set $\mathcal{L}\left(D_{k}:k\geq1\right)$ of partial limits of $\left(D_k:k\geq1\right)$ that was defined in \eqref{eq:15} is contained in the ratio set $\mathrm{e}\left(\mathscr{R}_A\right)$. In particular, if this partial limits set contains a positive length interval, or at least two numbers independent over the rationals, then $\mathrm{e}\left(\mathscr{\mathscr{R}}_{A}\right)=\mathbb{R}$.
\end{lem}

\begin{proof}
Let $r\in\mathcal{L}\left(D_{k}:k\geq1\right)$. Let $0<\epsilon<\min\left\{\left|r\right|,\delta\right\}$, where $\delta>0$ is the constant of the Doeblin condition. Let $E\in\sigma\left(X_k:\left|k\right|\leq N\right)$ for some $N\geq1$. Fix some large $K\geq1$ such that
$$j_{K}>N+M,\text{ }i_K<-N-M \text{ and }\left|D_K-r\right|<\epsilon,$$
where $M$ is such that $A^M>0$. Let
$$F:=B_K\left(i_{K}\right)\cap B_K'\left(j_{K}\right)\cap E\subset E$$
and consider the asymmetric admissible permutation $V:B_{K}\left(i_{K}\right)\rightleftharpoons B'_{K}\left(j_{k}\right)$. Then $V$ is a mapping of the form $V:F\to E$ and by Claim \ref{Claim: Admissible Permutation Derivative} it satisfies
$$\log V'\left(x\right)=D_K\in\left(r-\epsilon,r+\epsilon\right)\quad\text{for }x\in F.$$
Finally, since $j_K-N>M$ and $i_K+N<M$, we apply Proposition \ref{Proposition: Mixing in MSFT} twice to get
\begin{align*}
\mu\left(F\right)
&\geq C\left(\delta,M\right)^{2}\mu\left(B_K\left(i_{K}\right)\right)\mu\left(B_K'\left(j_{K}\right)\right)\mu\left(E\right)\\
& \qquad \geq C\left(\delta,M\right)^{2}\delta^{2\left(M+L\right)}\mu\left(E\right),
\end{align*}
where we used that in general, for every admissible block $B=\left[b_{1},\dotsc,b_{L}\right]$ and every $i\in\mathbb{Z}$, by Proposition \ref{Proposition: Mixing in MSFT} we have
\begin{align}
\label{eq:4}
\begin{split}
\mu\left(B\left(i\right)\right)
& =\pi_{i}\left(b_{1}\right)P_{i}\left(b_{1},b_{2}\right)\cdot\dotsm\cdot P_{i+L-1}\left(b_{L-1},b_{L}\right)\geq \delta^{M+L}.
\end{split}
\end{align}
This shows that the condition for extending $r$ to be an essential value of $\mathrm{e}\left(\mathscr{\mathscr{R}}_{A}\right)$ as in the Lemma \ref{Lemma: Approximating Essential Values}
is fulfilled for $\eta:=C\left(\delta,M\right)^{2}\delta^{2\left(M+L\right)}>0$.
\end{proof}

\begin{lem}
\label{Lemma: Connected Image}
Let $\mathcal{S}'\subset\mathcal{S}\times\mathcal{S}$ be some set with cardinality $d'$. Consider the set $\mathcal{L}\left(P'_n:n\geq1\right)$ of partial limits of the sequence
$$P'_n:=\left(P_n\left(s,t\right):\left(s,t\right)\in\mathcal{S}'\right)\in\left[\delta,1-\delta\right]^{d'},\quad n\geq1.$$
Then the image of every continuous real-valued function on $\mathcal{L}\left(P'_n:n\geq1\right)$ is a compact, connected set.
\end{lem}

\begin{proof}
A continuous real-valued function $f$ on $\mathcal{L}\left(P'_n:n\geq1\right)$ satisfies
$$f\left(\mathcal{L}\left(P'_{n}:n\geq1\right)\right)=\mathcal{L}\left(f\left(P'_{n}\right):n\geq1\right).$$
By Corollary \ref{Corollary: Nonsingularity of the Shift} we have
$$\mathrm{d}\left(P'_n,P'_{n-1}\right)\xrightarrow[n\to\infty]{}0$$ for some Euclidean metric $\mathrm{d}$ on $\left[\delta,1-\delta\right]^{d'}$, and since $f$ is uniformly continuous we have 
$$f\left(P'_n\right)-f\left(P'_{n-1}\right)\xrightarrow[n\to\infty]{}0.$$
Then the lemma follows from the following elementary fact. The partial limits set of a sequence $\left(p_n:n\geq1\right)$ of numbers with the property $p_n-p_{n-1}\xrightarrow[n\to\infty]{}0$ is a compact, connected set.
\end{proof}

\begin{lem}
\label{Lemma: Combinatorial Lemma}
Let $\mathcal{S}$ be a finite set. Let $P$ and $Q$ be a pair of different irreducible and aperiodic stochastic $\left|\mathcal{S}\right|\times\left|\mathcal{S}\right|$-matrices such that
$$P\left(s,t\right)=0\iff Q\left(s,t\right)=0,\quad s,t\in\mathcal{S}.$$
Then there is $L\geq1$ and a pair of elements $\alpha$ and $\beta$ in $\mathcal{S}$, as well as a pair of finite paths $\left[b_{1},\dots,b_{L}\right]$ and $\left[b'_{1},\dots,b'_{L}\right]$ in $\mathcal{S}$ that are admissible for $P$ (and $Q$), such that
$$\frac{P\left(\alpha,b_{1}\right)\dotsm P\left(b_{L},\beta\right)}{P\left(\alpha,b'_{1}\right)\dotsm P\left(b'_{L},\beta\right)}\neq\frac{Q\left(\alpha,b_{1}\right)\dotsm Q\left(b_{L},\beta\right)}{Q\left(\alpha,b'_{1}\right)\dotsm Q\left(b'_{L},\beta\right)}.$$
\end{lem}

Before the proof we establish some notations. For a matrix $P$ and a block $B=\left[b_{1},\dots,b_{L}\right]$ we write
$$P\left(B\right)=P\left(b_{1},b_{2}\right)\dotsm P\left(b_{L-1},b_{L}\right).$$
For a stochastic matrix $P$, consider the topologically-mixing SFT $X_A$ of $\mathcal{S}^{\mathbb{Z}}$ where $A$ is the $\left\{0,1\right\}$-valued $\left|\mathcal{S}\right|\times\left|\mathcal{S}\right|$-matrix defined by
$$A\left(s,t\right)=1\iff P\left(s,t\right)>0,\quad s,t\in\mathcal{S}.$$
For such $A$ and $P$ denote by $\mu_{P}$ the homogeneous Markov measure on $X_A$ defined by $P$ and its stationary distribution. For every $\alpha$ and $\beta$ in $\mathcal{S}$ and every integers $n<m$ denote by $\mathfrak{B}_{A}^{\left[n,m\right]}\left(\alpha,\beta\right)$ the finite collection of all $A$-admissible blocks on the coordinates $\left\{n,\dotsc,m\right\}$, who take the form $\left[\alpha,s_{n+1},\dotsc,s_{m-1},\beta\right]$ for some $s_{n+1},\dotsc,s_{m-1}$ in $\mathcal{S}$.

\begin{proof}[Proof of Lemma \ref{Lemma: Combinatorial Lemma}]
Suppose toward a contradiction that the assertion in the lemma is false. This means that
\begin{equation}
\label{eq:12}
P\left(B\right)Q\left(B'\right)=Q\left(B\right)P\left(B'\right)\text{ for every admissible pair }\left(B,B'\right)\text{ in }X_A,
\end{equation}
where $A$ is the adjacency matrix corresponding to $P$ (and $Q$). Consider the space $X_{A}\times X_{A}\subset\mathcal{S}^{\mathbb{Z}}\times\mathcal{S}^{\mathbb{Z}}$ with the Borel sigma-algebra $\mathcal{B}\left(X_{A}\times X_{A}\right)$ and let
$$\mathfrak{A}=\left\{ B\left(n\right)\times B'\left(n\right):\left(B,B'\right)\text{ is an admissible pair in }X_A\text{ and }n\in\mathbb{Z}\right\}.$$
We consider the two trivial ways to define joining on $X_A\times X_A$:
$$\mu_{\left(P,Q\right)}:=\mu_{P}\otimes\mu_{Q}\text{ and }\mu_{\left(Q,P\right)}:=\mu_{Q}\otimes\mu_{P}.$$
Then assumption \eqref{eq:12} means that $\mu_{\left(P,Q\right)}=\mu_{\left(Q,P\right)}$ on $\mathfrak{A}$.

\begin{claim}
\label{Claim1}
The product of the shifts
$$T\times T:X_{A}\times X_{A}\to X_{A}\times X_{A}$$
is ergodic with respect to $\mu_{\left(P,Q\right)}$.
\end{claim}

\begin{subproof}
It is well-known \cite[Corollary~1.1]{sarig2009lecture} that for an irreducible and aperiodic stochastic matrix $P$ the shift is (strongly-)mixing with respect to the Markov measure $\mu_P$. In our case, since the shift $T$ is mixing with respect to both $\mu_P$ and $\mu_Q$ it follows that $T\times T$ is ergodic with respect to $\mu_{\left(P,Q\right)}$ and the claim follows.
\end{subproof}

\begin{claim}
\label{Claim2}
$\mathfrak{A}$ is generating $\mathcal{B}\left(X_{A}\times X_{A}\right)$ up to $\mu_{\left(P,Q\right)}$-null sets.
\end{claim}

\begin{subproof}
Consider a basic cylinder $C_{0}\times C_{1}\in\mathcal{B}\left(X_{A}\right)\times\mathcal{B}\left(X_{A}\right)$ supported on the coordinates $\left\{-N,\dotsc,N\right\}\times\left\{-N,\dotsc,N\right\}$ for some $N\geq1$. Define stopping times $\tau_{+}$ and $\tau_{-}$ on $X_A\times X_A$ by
$$\tau_{+}\left(x,y\right)=\inf\left\{ n>N:x_{n}=y_{n}\right\} \text{ and }\tau_{-}\left(x,y\right)=\inf\left\{ n>N:x_{-n}=y_{-n}\right\}.$$
By Claim \ref{Claim1} $T\times T$ is ergodic with respect to $\mu_{\left(P,Q\right)}$, so by the pointwise ergodic theorem we have that
$$\lim_{K\to\infty}\frac{1}{K}\sum_{k=0}^{K-1}\mathbf{1}_{\left\{ x_{k}=y_{k}\right\} }=\mu_{\left(P,Q\right)}\left(x_{0}=y_{0}\right)>0\text{ for }\mu_{\left(P,Q\right)}\text{-a.e. }\left(x,y\right)\in X_{A}\times X_{A}.$$
This shows that $\tau_{+}<\infty$, $\mu_{\left(P,Q\right)}$-a.e. and similarly, by the ergodicity of $T^{-1}\times T^{-1}$, also $\tau_{-}<\infty$, $\mu_{\left(P,Q\right)}$-a.e. Observe that for every $s$ and $t$ in $\mathcal{S}$ and every choice of $B_{0}$ and $B_{1}$ in $\mathfrak{B}_{A}^{\left[\tau_{-},\tau_{+}\right]}\left(s,t\right)$ it holds that
$$\left(C_{0}\cap B_{0}\right)\times\left(C_{1}\cap B_{1}\right)\in\mathfrak{A}\text{ for }\mu_{\left(P,Q\right)}\text{-a.e. }\left(x,y\right)\in X_{A}\times X_{A}.$$
This shows that
$$C_{0}\times C_{1}=\bigcup_{s,t\in\mathcal{S}}\bigcup_{n<-N,N<m}\bigcup_{B_{0},B_{1}\in\mathfrak{B}_{A}^{\left[\tau_{-}=n,\tau_{+}=m\right]}\left(s,t\right)}\left(C_{0}\cap B_{0}\right)\times\left(C_{1}\cap B_{1}\right)$$
up to a $\mu_{\left(P,Q\right)}$-null set. This completes the proof of the claim.
\end{subproof}

\begin{claim}
\label{Claim3}
Every finite intersection of elements in $\mathfrak{A}$ is a \emph{disjoint} union of finitely many elements of $\mathfrak{A}$.
\end{claim}

\begin{subproof}
For some integers  $n_1\leq n_2$, let
$$\mathbf{B}_{1}:=B_{1}\left(n_{1}\right)\times B'_{1}\left(n_{1}\right)\text{ of length }L_{1}$$
and
$$\mathbf{B}_{2}:=B_{2}\left(n_{2}\right)\times B'_{2}\left(n_{2}\right)\text{ of length }L_{2}$$
be elements in $\mathfrak{A}$ such that $\mathbf{B}_{1}\cap\mathbf{B}_{2}$ is non-empty. If the sets of coordinates
$$\left\{n_{1},\dots,n_{1}+L_{1}\right\}\text{ and }\left\{n_{2},\dots,n_{2}+L_{2}\right\}$$
are not disjoint and $\mathbf{B}_{1}\cap\mathbf{B}_{2}$ is non-empty then simply $\mathbf{B}_{1}\cap\mathbf{B}_{2}\in\mathfrak{A}$. The same holds also if $n_1+L_1=n_2-1$.
Assume then that $n_1+L_1<n_2-1$. In this case we can write
$$B_{1}\left(n_{1}\right)\cap B_{2}\left(n_{2}\right)=\bigcup_{s,t\in\mathcal{S}}\bigcup_{B\in\mathfrak{B}_{A}^{\left[n_{1}+L_{1},n_{2}-1\right]}\left(s,t\right)}B_{1}\left(n_{1}\right)\ast B\ast B_{2}\left(n_{2}\right),$$
where some of the concatenated blocks may be empty. Of course, we can write $B'_{1}\left(n_{1}\right)\cap B'_{2}\left(n_{2}\right)$ in a similar way. Observe that
$$\left(B_{1}\left(n_{1}\right)\ast B\ast B_{2}\left(n_{2}\right),B'_{1}\left(n_{1}\right)\ast B'\ast B'_{2}\left(n_{2}\right)\right)$$
is an admissible pair on the coordinates $\left\{ n_{1},\dots,n_{2}+L_{2}\right\}$ for every choice of $B$ and $B'$ in $\mathfrak{B}_{A}^{\left[n_{1}+L_{1},n_{2}-1\right]}\left(s,t\right)$. Thus, as
$$\mathbf{B}_{1}\cap\mathbf{B}_{2}=\left(B_{1}\left(n_{1}\right)\cap B_{2}\left(n_{2}\right)\right)\times\left(B'_{1}\left(n_{1}\right)\cap B'_{2}\left(n_{2}\right)\right),$$
we conclude that it is a disjoint union of finitely many elements of $\mathfrak{A}$ and the proof of Claim \ref{Claim3} is complete.
\end{subproof}

\medskip

We will now prove that $\mu_{\left(P,Q\right)}=\mu_{\left(Q,P\right)}$ as measures on $\mathcal{B}\left(X_{A}\times X_{A}\right)$. Our argument is based on the Dynkin's π-λ theorem and we will follow the terminology of \cite[Chapter~II, $\mathsection$2]{shiryaev2013p}. Let
$$\mathcal{F}=\left\{ E\in\mathcal{B}\left(X_A\times X_A\right):\mu_{\left(P,Q\right)}\left(E\right)=\mu_{\left(Q,P\right)}\left(E\right)\right\}.$$
Then $\mathcal{F}$ is a d-system containing $\mathfrak{A}$. By Claim \ref{Claim3} if $\mathbf{B}_{1}$ and $\mathbf{B}_{2}$ are in $\mathfrak{A}$ then $\mathbf{B}_{1}\cap\mathbf{B}_{2}\in\mathcal{F}$. Thus, the π-system $\pi\left(\mathfrak{A}\right)$ generated by $\mathfrak{A}$, which consists of all finite intersections of elements of $\mathfrak{A}$, is also contained in $\mathcal{F}$. Then by the Dynkin's π-λ Theorem the sigma-algebra generated by $\pi\left(\mathfrak{A}\right)$ is contained in $\mathcal{F}$. By Claim \ref{Claim2} this sigma-algebra is $\mathcal{B}\left(X_{A}\times X_{A}\right)$ so that $\mathcal{B}\left(X_{A}\times X_{A}\right)=\mathcal{F}$.

Finally, to complete the proof of Lemma \ref{Lemma: Combinatorial Lemma}, we get a contradiction by showing that $P=Q$. For every $s\in\mathcal{S}$ let
$$\mathbf{B}=\bigcup_{t\in\mathcal{S}}\mathbf{B}_{t}\text{ for }\mathbf{B}_{t}=\left\{ \left(x,y\right)\in X_A\times X_A:x_{0}=s,y_{0}=t\right\},\, t\in\mathcal{S}.$$
Then $\mu_{\left(P,Q\right)}\left(\mathbf{B}\right)=\pi_{P}\left(s\right)$ and $\mu_{\left(Q,P\right)}\left(\mathbf{B}\right)=\pi_Q\left(s\right)$ so that $\pi_P\left(s\right)=\pi_Q\left(s\right)$ for all $s\in\mathcal{S}$. Next, for every $s,t\in\mathcal{S}$ let
$$\mathbf{B}=\bigcup_{u\in\mathcal{S}}\mathbf{B}_{u}\text{ for }\mathbf{B}_{u}=\left\{ \left(x,y\right)\in X_{A}\times X_{A}:\left(x_{0},y_{0}\right)=\left(s,t\right),\,\left(x_{1},y_{1}\right)=\left(s,u\right)\right\},\, u\in\mathcal{S}.$$
Then $\mu_{\left(P,Q\right)}\left(\mathbf{B}\right)=\pi_{P}\left(s\right)P\left(s,t\right)\pi_{Q}\left(s\right)$ and $\mu_{\left(Q,P\right)}\left(\mathbf{B}\right)=\pi_{Q}\left(s\right)Q\left(s,t\right)\pi_{P}\left(s\right)$. As $\pi_P\left(s\right)=\pi_Q\left(s\right)$ we see that $P\left(s,t\right)=Q\left(s,t\right)$.
\end{proof}

We are now in a position to prove Theorem \ref{Theorem: Theorem B}.

\begin{proof}[Proof of Theorem \ref{Theorem: Theorem B}]
By Theorem \ref{Theorem: Theorem A} we know that under the conditions of Theorem \ref{Theorem: Theorem B} the shift on $\left(X_A,\mu\right)$ is ergodic. Thus, by our Hopf Argument \ref{Theorem: Hopf Argument} if we show that $\mathrm{e}\left(\mathscr{R}_A\right)=\mathbb{R}$ for the renormalization full-group $\mathscr{R}_A:=\mathscr{R}\left(T;\Pi_A\right)$ it will follow that the shift is of type $\mathrm{III}_1$.

We consider the case where $\left(P_n:n\geq1\right)$ does not converge regardless the convergence of $\left(P_{-n}:n\geq1\right)$, and the other case is being similar. For the rest of the proof we fix an arbitrary sequence $i_{k}\xrightarrow[k\to\infty]{}-\infty$ of coordinates that satisfies $i_{k}-i_{k+1}\xrightarrow[k\to\infty]{}\infty$, such that $P_{i_{k}}\xrightarrow[k\to\infty]{}R$ for some arbitrary stochastic matrix $R$. For every partial limit $P_{j_{k}}\xrightarrow[k\to\infty]{}P$ for some $j_{k}\xrightarrow[k\to\infty]{}\infty$ and for every admissible pair of the form
\begin{equation}
\label{eq:13}
\left(B,B'\right)\text{ for }B=\left[b_{1},b_{2},\dots,b_{L-1},b_{L}\right],\,B'=\left[b_{1},b_{2}',\dots,b_{L-1}',b_{L}\right],
\end{equation}
assuming without loss of generality that $j_{k}-j_{k-1}>L+M$ for all $k\geq1$, the sequence of admissible permutations $V_{k}:B\left(i_{k}\right)\rightleftharpoons B'\left(j_{k}\right)$ satisfies
$$V_{k}'\left(x\right)=\frac{P_{i_{k}}\left(B'_{k}\right)P_{j_{k}}\left(B_{k}\right)}{P_{i_{k}}\left(B_{k}\right)P_{j_{k}}\left(B'_{k}\right)}\xrightarrow[k\to\infty]{}\frac{R\left(B'\right)}{R\left(B\right)}\cdot\frac{P\left(B\right)}{P\left(B'\right)},\quad x\in B\left(i_{k}\right)\cap B'\left(j_{k}\right),$$
where in this convergence we used the nonsingularity of the shift and Corollary \ref{Corollary: Nonsingularity of the Shift} to see that for every fixed $l\in\mathbb{Z}$ it holds that
$$\lim_{k\to\infty}\frac{P_{i_{k}+l}\left(s,t\right)}{P_{i_{k}}\left(s,t\right)}=1\text{ and }\lim_{k\to\infty}\frac{P_{j_{k}+l}\left(s,t\right)}{P_{j_{k}}\left(s,t\right)}=1\text{ for all }s,t\in\mathcal{S}.$$
Letting $c:=\log\left(R\left(B'\right)/R\left(B\right)\right)$ we see by Lemma \ref{Lemma: Partial Limits Essential Values} that
$$c+\log\frac{P\left(B\right)}{P\left(B'\right)}\in\mathrm{e}\left(\mathscr{R}\left(T;\Pi_{A}\right)\right).$$
Note that by the nonsingularity of the shift and Lemma \ref{Lemma: Connected Image}, for every admissible pair $\left(B,B'\right)$ of the form of \eqref{eq:13} the set of partial limits
$$\mathcal{L}\left(\left(P_{n}\left(b_{1},b_{2}\right),\dots,P_{n+L-2}\left(b_{L-1},b_{L}\right),P_{n}\left(b_{1},b_{2}'\right),\dots,P_{n+L-2}\left(b_{L-1}',b_{L}\right)\right):n\geq1\right)$$
is a compact, connected subset of $\left[\delta,1-\delta\right]^{2\left(L-1\right)}$. Denoting this partial limits set by $\mathcal{L}\left(B,B'\right)$, we see that the image $F\left(\mathcal{L}\left(B,B'\right)\right)$ of the set $\mathcal{L}\left(B,B'\right)$ under the continuous function $F:\left[\delta,1-\delta\right]^{2\left(L-1\right)}\to\mathbb{R}$ defined by
$$F:\left(r_{1},\dots,r_{L-1},r_{1}',\dots,r_{L-1}'\right)\mapsto\log\frac{r_{1}\dotsm r_{L-1}}{r_{1}'\dotsm r_{L-1}'},$$
is a compact, connected set of $\mathbb{R}$, which is simply a compact interval. By the above argument we have that $c+F\left(\mathcal{L}\left(B,B'\right)\right)\subset\mathrm{e}\left(\mathscr{R}_A\right)$. We then only need to show that there can be found some $L\geq3$ and an admissible pair $\left(B,B'\right)$ of length $L$ such that $F\left(\mathcal{L}\left(B,B'\right)\right)$ is an interval of positive length or, equivalently, that $F$ is not constant on $\mathcal{L}\left(B,B'\right)$. This is straightforward from Lemma \ref{Lemma: Combinatorial Lemma}.
\end{proof}

\section{Proof of the Necessary Condition for Conservativeness}

Here we prove Theorem \ref{Theorem: Theorem C}. First let us establish a general simple necessary condition for conservativeness. Let $\left(X,\mathcal{B},\mu\right)$ be a standard probability space and $T:X\to X$ an invertible bi-measurable transformation. Denote the ergodic sums of a function $f$ on $X$ by
$$S_{N}^{+}f\left(x\right)=\sum_{n=0}^{N-1}f\left(T^{n}x\right)\text{ and }S_{N}^{-}f\left(x\right)=\sum_{n=0}^{N-1}f\left(T^{-n}x\right)\text{ for }N\geq1.$$

\begin{lem}
\label{Lemma: Right-Left Limits}
In the above setting, if there is a real-valued function $f$ on $X$ such that for some numbers $a<b$ it holds that
$$\limsup_{N\to\infty}\frac{1}{N}S_{N}^{-}f\leq a<b\leq\liminf_{N\to\infty}\frac{1}{N}S_{N}^{+}f$$
on a set of $\mu$-positive measure, then $T$ is not conservative.
\end{lem}

\begin{proof}
Let $\epsilon=\left(b-a\right)/3$ and fix $N_0$ large enough such that the set
$$E_{0}:=\bigcap_{N\geq N_0}\left\{ \frac{1}{N}S_{N}^{-}f\leq a+\epsilon<b-\epsilon\leq\frac{1}{N}S_{N}^{+}f\right\}$$
is of $\mu$-positive measure. Then for every $x\in E_{0}$ and every $N\geq N_0$,
$$\frac{1}{N}S_{N}^{-}f\left(T^{N-1}x\right)=\frac{1}{N}S_{N}^{+}f\left(x\right)\geq b-\epsilon>a+\epsilon,$$
showing that $T^{N}x\notin E_{0}$ for all but at most finitely many positive integers $N$. Then by Halmos' Recurrence Theorem \cite[Chapter~1.1]{aaronson1997} $E_{0}$ is a $\mu$-positive measure set which is not in the conservative part of the shift.
\end{proof}

\begin{thm}[Wen--Weiguo \cite{wen1996, wen2004}]
\label{Theorem: Markov LLN}
Let $\left(X_{n}:n\geq0\right)$ be a non-homogeneous one-sided Markov chain and $\left(f_{n}:n\geq0\right)$ be a bounded sequence of functions on $\mathcal{S}\times\mathcal{S}$. Then
$$\frac{1}{N}\sum_{n=0}^{N-1}\left(f_{n}\left(X_{n},X_{n+1}\right)-\mathbf{E}\left(f_{n}\left(X_{n},X_{n+1}\right)\mid X_{n}\right)\right)\xrightarrow[N\to\infty]{\mathrm{a.e.}}0.$$
\end{thm}

Once we observe that
$$\xi_{n}:=f_{n}\left(X_{n},X_{n+1}\right)-\mathbf{E}\left(f_{n}\left(X_{n},X_{n+1}\right)\mid X_{n}\right),\quad n\geq1,$$
is a sequence of martingale differences for the natural filtration, Theorem \ref{Theorem: Markov LLN} follows from the Law of Large Numbers for martingales \cite[Theorem~2.19]{hall2014}.

Applying Theorem \ref{Theorem: Markov LLN} to the functions $\mathbf{1}_{\left\{ X_{n+1}=t_{0}\right\} }$, $t_{0}\in\mathcal{S}$ and to the functions $\mathbf{1}_{\left\{ \left(X_{n},X_{n+1}\right)=\left(s_{0},t_{0}\right)\right\} }$, $s_{0},t_{0}\in\mathcal{S}$ we get the following.

\begin{cor}
\label{Corollary: Markov LLN}
In the conditions of Theorem \ref{Theorem: Markov LLN}, we have
\begin{equation}
\frac{1}{N}\sum_{n=0}^{N-1}\left(\mathbf{1}_{\left\{ X_{n+1}=t_{0}\right\} }-P_{n}\left(X_{n},t_{0}\right)\right)\xrightarrow[N\to\infty]{\mathrm{a.e.}}0,\, t_0\in\mathcal{S},
\end{equation}
and
\begin{equation}
\frac{1}{N}\sum_{n=0}^{N-1}\left(\mathbf{1}_{\left\{ \left(X_{n},X_{n+1}\right)=\left(s_{0},t_{0}\right)\right\} }-\mathbf{1}_{\left\{ X_{n}=s_{0}\right\} }P_{n}\left(X_n,t_{0}\right)\right)\xrightarrow[N\to\infty]{\mathrm{a.e.}}0,\, s_0,t_0\in\mathcal{S}.
\end{equation}
\end{cor}

In the following discussion it will be useful to use the notation
$$W_{N}\approx W'_{N}\iff W_{N}-W'_{N}\xrightarrow[N\to\infty]{\mathrm{a.e.}}0,$$
which defines an equivalence relation on the collection of all sequences of random variables on a specified probability space.

The following proposition was proved by Wen and Weiguo \cite[Theorem~2]{wen2004} in the context of $m^{\text{th}}$ order Markov chains. In the following we provide a simplified version of their proof.

\begin{prop} (Wen--Weiguo)
\label{Lemma: Markov LLN}
Let $\left(X_{n}:n\geq 0\right)$ be a Markov chain with the distribution defined by $\left(\pi_{n},P_{n}:n\geq 0\right)$. If $P_{n}\left(s,t\right)\xrightarrow[n\to\infty]{}P\left(s,t\right)$ for all $s,t\in\mathcal{S}$ for an irreducible and aperiodic stochastic matrix $P$ with stationary distribution $\pi$, then
\begin{equation}
\label{eq:23}
\frac{1}{N}\sum_{n=0}^{N-1}\mathbf{1}_{\left\{ X_{n}=t_{0}\right\}}\approx\pi\left(t_{0}\right),\quad t_{0}\in\mathcal{S}
\end{equation}
and
\begin{equation}
\label{eq:24}
\frac{1}{N}\sum_{n=0}^{N-1}\mathbf{1}_{\left\{ \left(X_{n},X_{n+1}\right)=\left(s_{0},t_{0}\right)\right\}}\approx\pi\left(s_{0}\right)P\left(s_{0},t_{0}\right),\quad s_{0},t_{0}\in\mathcal{S}.
\end{equation}
\end{prop}

\begin{proof}
By Corollary \ref{Corollary: Markov LLN} and the Cesaro convergence for all $s_0,t_0\in\mathcal{S}$ it holds
$$\frac{1}{N}\sum_{n=0}^{N-1}\mathbf{1}_{\left\{ \left(X_{n},X_{n+1}\right)=\left(s_{0},t_{0}\right)\right\} } \approx P\left(s_{0},t_{0}\right)\frac{1}{N}\sum_{n=0}^{N-1}\mathbf{1}_{\left\{ X_{n}=s_{0}\right\} }.$$
Hence \eqref{eq:23} implies \eqref{eq:24}. To establish \eqref{eq:23}, fix $t_0\in\mathcal{S}$ and observe that by Corollary \ref{Corollary: Markov LLN} and the Cesaro convergence we have
$$\frac{1}{N}\sum_{n=0}^{N-1}\mathbf{1}_{\left\{ X_{n+1}=t_{0}\right\} } \approx\sum_{s\in\mathcal{S}}P\left(s,t_{0}\right)\frac{1}{N}\sum_{n=0}^{N-1}\mathbf{1}_{\left\{X_{n+1}=s\right\}}.$$
Denoting the $k$-fold product of $P$ by $P^k$ we get recursively that for every $k\geq1$,
$$\frac{1}{N}\sum_{n=0}^{N-1}\mathbf{1}_{\left\{ X_{n+1}=t_{0}\right\} }\approx\sum_{s\in\mathcal{S}}P^{k}\left(s,t_{0}\right)\frac{1}{N}\sum_{n=0}^{N-1}\mathbf{1}_{\left\{ X_{n+1}=s\right\} }.$$
Since $P$ is irreducible and aperiodic, by the convergence theorem for homogeneous Markov chains $P^{k}\left(s,t_0\right)\xrightarrow[k\to\infty]{}\pi\left(t_0\right)$ for all $s\in\mathcal{S}$ and \eqref{eq:23} follows.
\end{proof}

\begin{proof}[Proof of Theorem \ref{Theorem: Theorem C}]\
Let
$$P=\lim_{n\to\infty}P_{-n}\text{ and }Q=\lim_{n\to\infty}P_{n},$$
and denote their stationary distributions by $\pi$ and $\lambda$, respectively. Since the SFT $X_A$ is topologically-mixing so that all the entries of $A^M$ are positive, and since the matrices $\left(P_n:n\in\mathbb{Z}\right)$ satisfy the Doeblin condition, it is clear that all the entries of $P^M$ and $Q^M$ are positive, so that $P$ and $Q$ are irreducible and aperiodic. It follows from Proposition \ref{Lemma: Markov LLN} that
$$\frac{1}{N}\sum_{n=0}^{N-1}\mathbf{1}_{\left\{ X_{n}=s_{0}\right\} }\approx\lambda\left(s_{0}\right),\quad s_{0}\in\mathcal{S}.$$
Note that the reversed sequence $\left(X_{0},X_{-1},X_{-2}\dotsc\right)$ is a Markov chain with the transition matrices $\left(\pi_{n},\widehat{P}_{n}:n\leq0\right)$, where
$$\widehat{P}_{n}\left(s,t\right)=\frac{\pi_{n-1}\left(t\right)}{\pi_{n}\left(s\right)}P_{n-1}\left(t,s\right),\quad s,t\in\mathcal{S},\, n\leq0.$$
This sequence converges to $\widehat{P}\left(s,t\right):=\frac{\pi\left(t\right)}{\pi\left(s\right)}P\left(t,s\right)$. Note also that $P$ and $\widehat{P}$ share the same stationary distribution $\pi$ so by Proposition \ref{Lemma: Markov LLN} we have
$$\frac{1}{N}\sum_{n=0}^{N-1}\mathbf{1}_{\left\{ X_{-n}=s_{0}\right\} }\approx\pi\left(s_{0}\right),\quad s_{0}\in\mathcal{S}.$$
If $\pi\left(s_{0}\right)\neq\lambda\left(s_{0}\right)$ for some $s_{0}\in\mathcal{S}$, then applying Lemma \ref{Lemma: Right-Left Limits} to the function $f=\mathbf{1}_{\left\{ X_{0}=s_{0}\right\} }$ shows that the shift is not conservative. Assume then that $\pi=\lambda$. Fix $s_{0},t_{0}\in\mathcal{S}$ and let $f=\mathbf{1}_{\left\{ \left(X_{0},X_{1}\right)=\left(s_{0},t_{0}\right)\right\} }.$ By Proposition \ref{Lemma: Markov LLN} we have
$$\frac{1}{N}S_{N}^{+}f=\frac{1}{N}\sum_{n=0}^{N-1}\mathbf{1}_{\left\{ \left(X_{n},X_{n+1}\right)=\left(s_{0},t_{0}\right)\right\} }\approx\lambda\left(s_{0}\right)Q\left(s_{0},t_{0}\right)=\pi\left(s_{0}\right)Q\left(s_{0},t_{0}\right).$$
By the reasoning mentioned above, we can apply Proposition \ref{Lemma: Markov LLN} to the reversed chain to get that
$$\frac{1}{N}S_{N}^{-}f=\frac{1}{N}\sum_{n=0}^{N-1}\mathbf{1}_{\left\{ \left(X_{-n+1},X_{-n}\right)=\left(t_{0},s_{0}\right)\right\} }\approx\pi\left(t_{0}\right)\widehat{P}\left(t_{0},s_{0}\right)=\pi\left(s_{0}\right)P\left(s_{0},t_{0}\right).$$
Thus, if $P\left(s_{0},t_{0}\right)\neq Q\left(s_{0},t_{0}\right)$ for some $s_{0},t_{0}\in\mathcal{S}$, applying Lemma \ref{Lemma: Right-Left Limits} to the function $f=\mathbf{1}_{\left\{\left(X_{0},X_{1}\right)=\left(s_{0},t_{0}\right)\right\}}$ shows that the shift is not conservative.
\end{proof}

\section{Proofs of the Convergent Scenarios}

We start with a sufficient condition for the Central Limit Theorem (CLT).

\begin{thm} (Dobrushin)
\label{Theorem: Markov CLT}
Let $\left(Y_{n}:n\geq1\right)$ be a non-homogeneous Markov chain that its distribution satisfies the Doeblin condition \ref{Doeblin}. Let $\left(f_{n}:n\geq1\right)$ be a uniformly bounded sequence of real-valued functions. If
$$\sum_{n\geq1}\mathbf{V}\left(f_{n}\left(Y_{n}\right)\right)=\infty,$$
then the sequence $\left(f_n\left(Y_n\right):n\geq1\right)$ satisfies
$$\frac{S_{N}-\mathbf{E}\left(S_{N}\right)}{\sqrt{\mathbf{V}\left(S_{N}\right)}}\xrightarrow[N\to\infty]{\mathrm{d}}\mathcal{N},$$
where $S_{N}:=\sum_{n=1}^{N}f_{n}\left(Y_{n}\right)$ for $N\geq1$ and $\mathcal{N}$ is the standard normal distribution.
\end{thm}

This formulation is a special case of a sufficient condition for CLT established by Dobrushin \cite{dobrushin1956central}. See the formulation and the proof by Sethuraman and Varadhan \cite{sethuraman2005}. In their notations, the constants $C_n$ are uniformly bounded as $\left(f_{n}:n\geq1\right)$ is uniformly bounded, and the ergodic coefficients $\alpha_n$ are all in $\left[2\delta,1\right]$ by the Doeblin condition.

The following lemma is a Markovian version of what is sometimes called Araki--Woods Lemma \cite[Chapter~2]{danilenko2019}.
 
\begin{lem}
\label{Lemma: Convergent Case}
Let $\left(X_A,\mu\right)$ be a topologically-mixing MSFT that satisfies the Doeblin condition \ref{Doeblin}. If there is an admissible configuration $\left(B_k\left(i_k\right),B'_k\left(j_k\right)\right)$, $k\geq1$, (recall Definition \ref{Definition: Admissible Configuration}) such that for the corresponding sequence $\left(D_{k}:k\geq1\right)$ defined in \eqref{eq:15} it holds that
$$D_{k}\xrightarrow[k\to\infty]{}0\quad\text{and}\quad\sum_{k\geq1}D_{k}^{2}=\infty,$$
then $\mathrm{e}\left(\mathscr{\mathscr{R}}_{A}\right)=\mathbb{R}$ for the renormalization full-group $\mathscr{\mathscr{R}}_{A}:=\mathscr{\mathscr{R}}\left(T;\Pi_{A}\right)$.
\end{lem}

\begin{proof}
Consider the sequence of symmetric admissible permutations
$$V_{k}:B_{k}\left(i_{k}\right)\circlearrowleft B_{k}'\left(j_{k}\right),\quad k\geq1.$$
Let the random variables
\begin{equation}
\label{eq:14}
Y_{k}\left(x\right):=\mathbf{1}_{B_{k}\left(i_{k}\right)\cap B_{k}'\left(j_{k}\right)}\left(x\right)-\mathbf{1}_{B'_{k}\left(i_{k}\right)\cap B_{k}\left(j_{k}\right)}\left(x\right),\quad k\geq1,
\end{equation}
so that according to Claim \ref{Claim: Admissible Permutation Derivative} and the notation in \eqref{eq:15},
$$\log V_k'\left(x\right)=D_kY_k\left(x\right),\quad k\geq1.$$

\begin{claim}
\label{Claim4}
The sequence $\left(Y_{k}:k\geq1\right)$ defined in \eqref{eq:14} is a one-sided Markov chain on the state space $\left\{-1,0,1\right\}$, with respect to the distribution induced from $\mu$ in the obvious way. Moreover, if $\mu$ satisfies the Doeblin condition for $\delta>0$ then the distribution of $\left(Y_k:k\geq1\right)$ satisfies the Doeblin condition for some $\delta'>0$.
\end{claim}

\begin{subproof}
Since the distribution $\mu$ of $\left(X_n:n\in\mathbb{Z}\right)$ satisfies the Markov property, it follows from \cite[Remark~10.9]{georgii2011} that $\mu$ satisfies the \textit{Markov field} property, namely
$$\sigma\left(X_{k}:\left|k\right|>n\right)\text{ conditioned on }\sigma\left(X_{-n},X_{n}\right)\text{ is independent on }\sigma\left(X_{k}:\left|k\right|<n\right)$$
for every $n\geq1$ with respect to $\mu$. This readily implies that the distribution of $\left(Y_k:k\geq1\right)$ satisfies the Markov property.

To see that the Markov chain $\left(Y_k:k\geq1\right)$ satisfies Doeblin condition we use Proposition \ref{Proposition: Mixing in MSFT}, and that by the construction of an admissible configuration, $i_{k}-i_{k+1}$ and $j_{k+1}-j_{k}$, as well as $j_k-i_k$, are all greater then $L+M$ for every $k\geq1$. First recall that for every $A$-admissible block $B$ of length $L$, for every $n\in\mathbb{Z}$ it holds that $\delta^{LM}\leq\mu\left(B\left(n\right)\right)\leq\left(1-\delta^{M}\right)^{L}$. We then get that
\begin{align*}
\mathbb{P}\left(Y_{k}=1\right)
&=\mu\left(B_{k}\left(i_{k}\right)\cap B_{k}'\left(j_{k}\right)\right)\\
&\geq C\left(\delta,M\right)\mu\left(B_{k}\left(i_{k}\right)\right)\mu\left(B_{k}'\left(j_{k}\right)\right)\\
&\geq C\left(\delta,M\right)\delta^{2LM},
\end{align*}
and similarly $\mathbb{P}\left(Y_{k}=-1\right)\geq C\left(\delta,M\right)\delta^{2LM}$. Also we get that
\begin{align*}
\mathbb{P}\left(Y_{k}=0\right)
&\geq\mu\left(B_{k}\left(i_{k}\right)^{\mathsf{c}}\cap B_{k}\left(j_{k}\right)^{\mathsf{c}}\right)\\
&\geq C\left(\delta,M\right)\mu\left(B_{k}\left(i_{k}\right)^{\mathsf{c}}\right)\mu\left(B_{k}\left(j_{k}\right)^{\mathsf{c}}\right)\\
&\geq C\left(\delta,M\right)\left(1-\left(1-\delta^{M}\right)^{L}\right)^{2}.
\end{align*}
Then let $0<\eta\leq1/2$ such that for all $k\geq1$ and $a\in\left\{ -1,0,1\right\}$ it holds that $\mathbb{P}\left(Y_{k}=a\right)\geq\eta$. Considering the transition probabilities, using the same considerations we see that for all $a,b\in\left\{-1,0,1\right\}$,
$$\mathbb{P}\left(Y_{k+1}=b,Y_{k}=a\right)\geq C\left(\delta,M\right)^{3}\eta^{4},$$
hence
$$\mathbb{P}\left(Y_{k+1}=b\mid Y_{k}=a\right)\geq\delta':=\frac{C\left(\delta,M\right)^{2}\eta^{3}}{1-\eta},$$
which concludes that the Markov chain $\left(Y_k:k\geq1\right)$ satisfies the Doeblin condition for $\delta'>0$
depending only on the constants $\delta$, $M$ and $L$. This completes the proof of Claim \ref{Claim4}.
\end{subproof}

\begin{claim}
\label{Claim5}
The sequence $\left(\log V_k':k\geq1\right)$ satisfies the central limit theorem.
\end{claim}

\begin{subproof}
As we mentioned, we have the identity $\log V_k'\left(x\right)=D_kY_k\left(x\right)$ for all $k\geq1$, where $\left(D_k:k\geq1\right)$ is a convergent sequence of numbers. It follows by Claim \ref{Claim4} that $\left(\log V_k':k\geq1\right)$ is a Markov chain that satisfies the Doeblin condition. To use Theorem \ref{Theorem: Markov CLT}, note that for every $k\geq1$ the events $B_{k}\left(i_{k}\right)\cap B_{k}'\left(j_{k}\right)$ and $B_{k}'\left(i_{k}\right)\cap B_{k}\left(j_{k}\right)$ are disjoint, so by Proposition \ref{Proposition: Mixing in MSFT} we have that
\begin{align*}
\mathbf{V}\left(Y_{k}\right)&
=\mu\left(B_{k}\left(i_{k}\right)\cap B_{k}'\left(j_{k}\right)\right)\left(1-\mu\left(B_{k}\left(i_{k}\right)\cap B_{k}'\left(j_{k}\right)\right)\right)\\
&\quad+\mu\left(B_{k}'\left(i_{k}\right)\cap B_{k}\left(j_{k}\right)\right)\left(1-\mu\left(B_{k}'\left(i_{k}\right)\cap B_{k}\left(j_{k}\right)\right)\right)\\
&\quad+2\mu\left(B_{k}\left(i_{k}\right)\cap B_{k}'\left(j_{k}\right)\right)\mu\left(B_{k}'\left(i_{k}\right)\cap B_{k}\left(j_{k}\right)\right)\\
&\geq\mu\left(B_{k}\left(i_{k}\right)\cap B_{k}'\left(j_{k}\right)\right)\geq C\left(\delta,M\right)\delta^{2LM},
\end{align*}
hence
$$\mathbf{V}\left(\log V_{k}'\right)=D_{k}^2\mathbf{V}\left(Y_k\right)\asymp D_{k}^2.$$
By the assumption in the Lemma we conclude that $\sum_{k\geq1}\mathbf{V}\left(\log V_{k}'\right)=\infty$, hence the sequence $\left(\log V_k':k\geq1\right)$ satisfies the condition of Theorem \ref{Theorem: Markov CLT}. This completes the proof of Claim \ref{Claim5}.
\end{subproof}

\begin{claim}
\label{Claim6}
For integers $1\leq k\leq K$ denote
$$S_{k}^{K}\left(x\right)=\sum_{i=k}^{K}\log V_{i}'\left(x\right).$$
Fix some $k_0\geq1$. Then for every $r<0$ it holds that
$$\liminf_{K\to\infty}\mathbb{P}\left(S_{k_{0}}^{K}<r\right)\geq\lim_{K\to\infty}\mathbb{P}\left(S_{k_{0}}^{K}<\mathbf{E}\left(S_{k_{0}}^{K}\right)\right)=1/2.$$
\end{claim}

\begin{subproof}
The second equality is a straightforward corollary of the CLT as in Claim \ref{Claim5}. For the first inequality, it is enough to show that $\mathbf{E}\left(S_{k_{0}}^{K}\right)\xrightarrow[K\to\infty]{}-\infty$. Note that since every $\left(B_k,B'_k\right)$ is an admissible pair, if we denote the mutual first symbol by $b_0$ and the mutual last symbol by $b_1$, then we have that
$$\mu\left(B_{k}\left(i_{k}\right)\cap B'_{k}\left(j_{k}\right)\right)=\pi_{i_{k}}\left(b_{0}\right)P_{i_{k}}\left(B_{k}\right)P^{\left(i_{k}+L-1,j_k\right)}\left(b_{1},b_{0}\right)P_{j_{k}}\left(B'_{k}\right)$$
and
$$\mu\left(B'_{k}\left(i_{k}\right)\cap B{}_{k}\left(j_{k}\right)\right)=\pi_{i_{k}}\left(b_{0}\right)P_{i_{k}}\left(B'_{k}\right)P^{\left(i_{k}+L-1,j_k\right)}\left(b_{1},b_{0}\right)P_{j_{k}}\left(B_{k}\right).$$
It then follows that
\begin{align*}
\mathbf{E}\left(\log V_{k}'\right)
&= D_k\mathbf{E}\left(Y_{k}\right)\\
&= D_{k}\left(\mu\left(B_{k}\left(i_{k}\right)\cap B'_{k}\left(j_{k}\right)\right)-\mu\left(B'_{k}\left(i_{k}\right)\cap B{}_{k}\left(j_{k}\right)\right)\right)\\
&\asymp D_{k}\left(P_{i_{k}}\left(B_{k}\right)P_{j_{k}}\left(B'_{k}\right)-P_{i_{k}}\left(B'_{k}\right)P_{j_{k}}\left(B_{k}\right)\right)\\
&\asymp-D_{k}^{2},
\end{align*}
where the first approximation is by the above calculation and Proposition \ref{Proposition: Mixing in MSFT}, and the second approximation was mentioned in \eqref{eq:15}. By the assumption in the Lemma we conclude that $\mathbf{E}\left(S_{k_{0}}^{K}\right)\xrightarrow[K\to\infty]{}-\infty$, which completes the proof of Claim \ref{Claim6}.
\end{subproof}

We are now ready to establish that $\mathrm{e}\left(\mathscr{R}_A\right)=\mathbb{R}$. Since the ratio set is an additive subgroup of $\mathbb{R}$ it is enough to show that it contains every negative number. Let $r<0$ and $0<\epsilon<\left|r\right|$. Let $E\in\sigma\left(X_k:\left|k\right|\leq N\right)$ for some $N\geq1$. Find positive integers $k_{0}\leq K_{0}$ to satisfy the following properties.
\begin{enumerate}
	\item $i_{k_{0}}+N\leq-M$ and $j_{k_{0}}-N\geq M$;
	\item $\left|\log V_k'\left(x\right)\right|<\epsilon$ everywhere for all $k\geq k_{0}$; and
	\item $\mu\left(S_{k_{0}}^{K_{0}}<r\right)\geq1/4$.
\end{enumerate}
The first property clearly holds for every large $k_0$. The second property holds for every large $k_0$ since $D_k\xrightarrow[k\to\infty]{}0$. The third property holds for every large $k_0$ and every $K_0$ which is large enough with respect to the choice of $k_0$ by Claim \ref{Claim6}.

Consider the set $F:=E\cap\left\{ S_{k_{0}}^{K_{0}}<r\right\} \subset E$. We now define $V\in\left[\left[\mathscr{R}_A\right]\right]$ of the form $V:F\to E$ with $\log V'\left(x\right)\in\left(r-\epsilon,r+\epsilon\right)$ for all $x\in F$. For $x\in F$ let
$$\text{k}\left(x\right):=\inf\left\{ k\geq k_{0}:S_{k_{0}}^{k}\left(x\right)<r\right\} \leq K_{0}$$
and
$$\mathrm{K}\left(x\right):=\left\{k_0\leq k\leq\mathrm{k}\left(x\right):Y_k\left(x\right)\neq0\right\}\subset\left\{k_0,\dotsc,\mathrm{k}\left(x\right)\right\}.$$
Define $Vx$ for $x\in F$ to be the composition of all $V_kx$ for $k\in\mathrm{K}\left(x\right)$.

Recall that $j_{k+1}-j_k$ and $i_k-i_{k+1}$ are both greater then $L+M$ for all $k\geq1$, and in particular the coordinates that are being changed by the $V_k$'s are distinct, so that $Vx$ is a well-defined transformation with domain in $X_A$. Also note that $Vx\in E$ for all $x\in F$, since $i_{k_{0}}+N\leq-M$ and $j_{k_{0}}-N\geq M$ while $E\in\sigma\left(X_k:\left|k\right|\leq N\right)$. We show that $V$ is one-to-one on $F$. Assume that $Vx=Vy$ for $x,y\in F$. If $\mathrm{k}\left(x\right)<\mathrm{k}\left(y\right)$, since $Vx=Vy$ implies that $x_k=y_k$ for all $\left|k\right|\leq\mathrm{k}\left(x\right)$, we get that
$$S_{k_{0}}^{\mathrm{k}\left(x\right)}\left(y\right)=S_{k_{0}}^{\mathrm{k}\left(x\right)}\left(x\right)<r,$$
a contradiction to the definition of $\mathrm{k}\left(y\right)$. By the symmetric reasoning it is also impossible that $\mathrm{k}\left(x\right)>\mathrm{k}\left(y\right)$, hence $\mathrm{k}\left(x\right)=\mathrm{k}\left(y\right)$. Then we see that for every $k_0\leq k\leq \mathrm{k}\left(x\right)=\mathrm{k}\left(y\right)$,
\begin{align*}
	& x\in B_{k}\left(i_{k}\right)\cap B_{k}'\left(j_{k}\right)\\
	& \quad\iff Vy=Vx\in B_{k}'\left(i_{k}\right)\cap B_{k}\left(j_{k}\right)\\
	&\quad\iff y\in B_{k}\left(i_{k}\right)\cap B_{k}'\left(j_{k}\right),
\end{align*}
and similarly
$$x\in B_{k}'\left(i_{k}\right)\cap B_{k}\left(j_{k}\right)\iff y\in B_{k}'\left(i_{k}\right)\cap B_{k}\left(j_{k}\right).$$
It follows that $\mathrm{K}\left(x\right)=\mathrm{K}\left(y\right)$. Finally, since each of the $V_{k}$'s is one-to-one and since $Vx=Vy$ is the composition of all $V_{k}$'s for $k\in\mathrm{K}\left(x\right)=\mathrm{K}\left(y\right)$, we see that $x=y$ so that $V$ is one-to-one on $F$. We also see that for every $x\in F$,
$$\log V'\left(x\right)=\sum_{k\in\mathrm{K}\left(x\right)}\log V_{k}'\left(x\right)=S_{k_{0}}^{\mathrm{k}\left(x\right)}\left(x\right)\in\left(r-\epsilon,r\right),$$
by the definition of $\text{k}\left(x\right)$ and since $\left|\log V_k'\left(x\right)\right|<\epsilon$ for $k\geq k_{0}$. This establishes the condition for $r$ to be essential value for all sets supported on finitely many coordinates. As this collection forms a countable algebra that is dense in the Borel sigma-algebra, in order to finish we establish the condition of Lemma \ref{Lemma: Approximating Essential Values}. Note that since $i_{k_{0}}+N\leq-M$ and $j_{k_{0}}-N\geq M$ while $E\in\sigma\left(X_k:\left|k\right|\leq N\right)$, by Proposition \ref{Proposition: Mixing in MSFT} we have that
$$\mu\left(F\right)\geq C\left(\delta,M\right)^2\mu\left(S_{k_{0}}^{K_{0}}<r\right)\mu\left(E\right)\geq\frac{C\left(\delta,M\right)^2}{4}\mu\left(E\right).$$
Thus, $\eta:=C\left(\delta,M\right)^2/4$ satisfies the condition of Lemma \ref{Lemma: Approximating Essential Values} and $r\in\mathrm{e}\left(\mathscr{\mathscr{R}}_{A}\right)$.
\end{proof}

We are now in a position to prove Theorems \ref{Theorem: Theorem D} and \ref{Theorem: Theorem E}. By Theorem \ref{Theorem: Theorem A} we know that under the conditions of Theorems \ref{Theorem: Theorem D} and \ref{Theorem: Theorem E} the shift is ergodic. Thus, by our Hopf Argument \ref{Theorem: Hopf Argument} if we show that $\mathrm{e}\left(\mathscr{R}_A\right)=\mathbb{R}$ for the renormalization full-group $\mathscr{R}_A=\mathscr{R}\left(T;\Pi_A\right)$ it will follow that the shift is of type $\mathrm{III}_1$.

Note that if
$$\sum_{n\geq1}\sum_{s,t\in\mathcal{S}}\left(\sqrt{P_n\left(s,t\right)}-\sqrt{Q\left(s,t\right)}\right)^{2}<\infty$$
then by Theorem \ref{Theorem: Kakutani Criteria} we can assume that $P_n=Q$ for all $n\geq1$ without changing the equivalence class of the measure. In a similar way, if
$$\sum_{n\geq1}\sum_{s,t\in\mathcal{S}}\left(\sqrt{\widehat{P}_{-n}\left(s,t\right)}-\sqrt{\widehat{Q}\left(s,t\right)}\right)^{2}<\infty$$
we can assume without loss of generality that $\widehat{P}_{-n}=\widehat{Q}$ for $n\geq1$. Then if both of those series are finite, $\mu$ is equivalent to the homogeneous Markov measure defined by $Q$ and the shift is of type $\mathrm{II}_1$. Thus, if the shift is not of type $\mathrm{II}_1$ then one of the above series diverges. We consider the case where the first series diverges,
$$\sum_{n\geq1}\sum_{s,t\in\mathcal{S}}\left(\sqrt{Q\left(s,t\right)}-\sqrt{P_{n}\left(s,t\right)}\right)^{2}=\infty,$$
and the other case is being similar.

\begin{proof}[Proof of Theorem \ref{Theorem: Theorem D}]

Let us consider first the fullshift $T:\left(X,\mu\right)\to\left(X,\mu\right)$ where $X=\left\{ 0,1\right\}^{\mathbb{Z}}$ and $\mu$ is the Markov measure defined by the transition matrices
$$P_{n}=\left(\begin{array}{cc}
p_{n} & 1-p_{n}\\
p'_{n} & 1-p'_{n}
\end{array}\right),\quad n\in\mathbb{Z}.$$
Assume that $\lim_{\left|n\right|\to\infty}P_{n}=Q$ for
$$Q=\left(\begin{array}{cc}
q & 1-q\\
q' & 1-q'
\end{array}\right).$$
If the shift is not of type $\mathrm{II}_{1}$ we assume without loss of generality that
\begin{align*}
&\sum_{n\geq1}\left(\sqrt{p_{n}}-\sqrt{q}\right)^{2}+\left(\sqrt{1-p_{n}}-\sqrt{1-q}\right)^{2}\\
&\qquad+\sum_{n\geq1}\left(\sqrt{p'_{n}}-\sqrt{q'}\right)^{2}+\left(\sqrt{1-p'_{n}}-\sqrt{1-q'}\right)^{2}=\infty.
\end{align*}
By the Doeblin condition the square roots do not affect this divergence so that
$$\sum_{n\geq1}\left(p_{n}-q\right)^{2}+\sum_{n\geq1}\left(p'_{n}-q'\right)^{2}=\infty.$$
We consider the case of $\sum_{n\geq1}\left(p_{n}-q\right)^{2}=\infty$ regardless $\sum_{n\geq1}\left(p'_{n}-q'\right)^{2}$. The other case can be treated symmetrically.

We now construct an admissible configuration that is satisfying the conditions of Lemma \ref{Lemma: Convergent Case}. Let
$$I:=\left\{ n\geq1:\mathrm{sign}\left(p_{n}-q\right)=\mathrm{sign}\left(p_{n+1}-q\right)\right\}.$$
Then one can easily see that $\left(p_{n}-p_{n+1}\right)^{2}\geq\left(p_{n}-q\right)^{2}$ for all $n\notin I$ so by the nonsingularity of the shift, using Corollary \ref{Corollary: Nonsingularity of the Shift} we get that
$$\sum_{n\notin I}\left(p_{n}-q\right)^{2}\leq\sum_{n\notin I}\left(p_{n}-p_{n+1}\right)^{2}<\infty.$$
It follows that
$$\sum_{n\in I}\left(p_{n}-q\right)^{2}=\infty.$$
Then we can find a subsequence $\left(j_{k}:k\geq1\right)\subset I$ satisfying
\begin{itemize}
	\item $j_{k}-j_{k-1}\geq3$ for all $k\geq1$;
	\item $\mathrm{sign}\left(p_{j_{k}}-q\right)=\mathrm{sign}\left(p_{j_{k}+1}-q\right)$ is constant for $k\geq1$; and
	\item $\sum_{k\geq1}\left(p_{j_{k}}-q\right)^{2}=\infty.$
\end{itemize}
Write $s:=\mathrm{sign}\left(p_{j_k}-q\right)$ for any $k\geq1$. Since $p_{n}\xrightarrow[n\to-\infty]{}q$ there is a sequence $\left(i_{k}:k\geq1\right)$ of negative integers satisfying
\begin{itemize}
	\item $i_{k}-i_{k+1}\geq3$ for all $k\geq1$;
	\item $\mathrm{sign}\left(p_{j_{k}}-p_{i_{k}}\right)=s$ for all $k\geq1$; and
	\item $\sum_{k\geq1}\left(p_{j_{k}}-p_{i_{k}}\right)^{2}=\infty.$
\end{itemize}
Consider the admissible pairs
$$\left(B_{0},B_{0}'\right)\text{ for }B_{0}=\left[0,0,0\right],\,B_{0}'=\left[0,1,0\right],$$
and
$$\left(B_{1},B_{1}'\right)\text{ for }B_{1}=\left[0,0,1\right],\,B_{1}'=\left[0,1,1\right].$$
For every $k\geq1$ we have that
\begin{align*}
D_{0,k}
&:=\log\left(\frac{P_{i_{k}}\left(B_{0}'\right)P_{j_{k}}\left(B_{0}\right)}{P_{i_{k}}\left(B_{0}\right)P_{j_{k}}\left(B_{0}'\right)}\right)\\
&=\log\frac{p_{j_{k}}}{p_{i_{k}}}+\log\frac{1-p_{i_{k}}}{1-p_{j_{k}}}+\log\frac{p_{j_{k}+1}}{p_{i_{k}+1}}+\log\frac{p'_{i_{k}+1}}{p'_{j_{k}+1}}
\end{align*}
and that
\begin{align*}
D_{1,k}
&:=\log\left(\frac{P_{i_{k}}\left(B_{1}'\right)P_{j_{k}}\left(B_{1}\right)}{P_{i_{k}}\left(B_{1}\right)P_{j_{k}}\left(B_{1}'\right)}\right)\\
&=\log\frac{p_{j_{k}}}{p_{i_{k}}}+\log\frac{1-p_{i_{k}}}{1-p_{j_{k}}}+\log\frac{1-p_{j_{k}+1}}{1-p_{i_{k}+1}}+\log\frac{1-p'_{i_{k}+1}}{1-p'_{j_{k}+1}}.
\end{align*}
Define for $k\geq1$,
$$g\left(k\right)=\begin{cases}
0 & \mathrm{sign}\left(\log\frac{p'_{i_{k}+1}}{p'_{j_{k}+1}}\right)=s\\
1 & \mathrm{sign}\left(\log\frac{1-p'_{i_{k}+1}}{1-p'_{j_{k}+1}}\right)=s
\end{cases}.$$

\begin{claim}
\label{Claim: Convergent Fullshift}
Let $D_k:=D_{g\left(k\right),k}$ for $k\geq1$. Then
$$D_k\xrightarrow[k\to\infty]{}0\quad\text{and }\sum_{k\geq1}D_k^2=\infty.$$
\end{claim}

\begin{subproof}[Proof of Claim \ref{Claim: Convergent Fullshift}]
It is clear that $D_k\xrightarrow[k\to\infty]{}0$. We prove the second part. By the approximation in \ref{Fact: Log Approximation} we have that
$$\log\frac{p_{j_{k}}}{p_{i_{k}}}\asymp\log\frac{1-p_{i_{k}}}{1-p_{j_{k}}}\asymp\log\frac{p_{j_{k}+1}}{p_{i_{k}+1}}\asymp p_{j_{k}}-p_{i_{k}}.$$
By the definition of $I$ we have that
$$\mathrm{sign}\left(\log\frac{1-p_{i_{k}}}{1-p_{j_{k}}}\right)=\mathrm{sign}\left(\log\frac{p_{j_{k}+1}}{p_{i_{k}+1}}\right)=s,\quad k\geq1.$$
It follows that for $g\left(k\right)=0$ we have
$$D_{k}=D_{0,k}\succcurlyeq p_{i_{k}}-p_{j_{k}},$$
so in case of $\sum_{g\left(k\right)=0}\left(p_{j_{k}}-p_{i_{k}}\right)^{2}=\infty$ we have
$$\sum_{k\geq1}D_{k}^{2}\geq\sum_{g\left(k\right)=0}D_{k}^{2}=\infty.$$
In case of $\sum_{g\left(k\right)=0}\left(p_{j_{k}}-p_{i_{k}}\right)^{2}<\infty$ we must have $\sum_{g\left(k\right)=1}\left(p_{j_{k}}-p_{i_{k}}\right)^{2}=\infty$. For $g\left(k\right)=1$ the general term of the sequence $D_k=D_{1,k}$ is the sum of the general term of the sequence
$$\log\frac{p_{j_{k}}}{p_{i_{k}}}+\log\frac{1-p'_{i_{k}+1}}{1-p'_{j_{k}+1}}\succcurlyeq p_{j_{k}}-p_{i_{k}}$$
and the general term of the sequence
$$\log\frac{1-p_{i_{k}}}{1-p_{j_{k}}}+\log\frac{1-p_{j_{k}+1}}{1-p_{i_{k}+1}}\asymp\left(p_{i_{k}+1}-p_{i_{k}}\right)+\left(p_{j_{k}}-p_{j_{k}+1}\right),$$
which is square-summable by the nonsingularity of the shift as in Corollary \ref{Corollary: Nonsingularity of the Shift}. It follows that also in this case we have
$$\sum_{k\geq1}D_{k}^{2}\geq\sum_{g\left(k\right)=1}D_{k}^{2}=\infty,$$
so the proof of Claim \ref{Claim: Convergent Fullshift} is complete.
\end{subproof}

By Claim \ref{Claim: Convergent Fullshift} we see that the admissible configuration $\left(B_{g\left(k\right)}\left(i_k\right),B'_{g\left(k\right)}\left(j_k\right)\right)$, $k\geq1$, satisfies the conditions of Lemma \ref{Lemma: Convergent Case} so that $\mathrm{e}\left(\mathscr{\mathscr{R}}_{A}\right)=\mathbb{R}$. This completes the proof of Theorem \ref{Theorem: Theorem D} for the fullshift.

\vspace{5mm}

Let us now consider a subshift on two states. The primitive adjacency matrices in this case, except from the fullshift, are
$$\left(\begin{array}{cc}
1 & 1\\
1 & 0
\end{array}\right)\text{ and }\left(\begin{array}{cc}
0 & 1\\
1 & 1
\end{array}\right).$$
The treatment in these two subshifts is similar, and we consider the first one. Let the transition matrices
$$P_{n}=\left(\begin{array}{cc}
p_{n} & 1-p_{n}\\
1 & 0
\end{array}\right),\quad n\in\mathbb{Z}.$$
Assume that $\lim_{\left|n\right|\to\infty}P_{n}=Q$ for
$$Q=\left(\begin{array}{cc}
q & 1-q\\
1 & 0
\end{array}\right).$$
If the shift is not of type $\mathrm{II}_1$ then without loss of generality $\sum_{n\geq1}\left(p_{n}-q\right)^{2}=\infty$. Choose sequences $\left(i_k:k\geq1\right)$ and $\left(j_k:k\geq1\right)$ in the same way we did in the fullshift and consider the admissible pair
$$\left(B,B'\right)\text{ for }B=\left[0,1,0\right],\,B'=\left[0,0,0\right].$$
Then $\left(B\left(i_k\right),B'\left(j_k\right)\right)$, $k\geq1$, is an admissible configuration that satisfies
\begin{align*}
D_{k}
&:=\log\left(\frac{P_{i_{k}}\left(B'\right)P_{j_{k}}\left(B\right)}{P_{i_{k}}\left(B\right)P_{j_{k}}\left(B'\right)}\right)\\
&=\log\frac{p_{i_{k}}}{p_{j_{k}}}+\log\frac{1-p_{j_{k}}}{1-p_{i_{k}}}+\log\frac{p_{i_{k}+1}}{p_{j_{k}+1}}\asymp p_{i_k}-p_{j_k},
\end{align*}
as $\mathrm{sign}\left(p_{i_{k}}-p_{j_{k}}\right)=\mathrm{sign}\left(p_{i_{k}+1}-p_{j_{k}+1}\right)$ is constant for $k\geq1$. Then $\sum_{k\geq1}D_{k}^{2}=\infty$ and clearly $D_k\xrightarrow[k\to\infty]{}0$. By Lemma \ref{Lemma: Convergent Case} we conclude that $\mathrm{e}\left(\mathscr{\mathscr{R}}_{A}\right)=\mathbb{R}$.
\end{proof}

\begin{rem}
In a similar way, one may prove the Bernoulli case on a general finite state space that is partially extending the results of \cite{kosloff2014,danilenko2019} concerning the half-stationary two-state space. Note that their results holds also without assuming the Doeblin condition. Consider the following setting. Let $\mathcal{S}=\left\{ 0,1,\dots,d-1\right\}$  for some $d\in\mathbb{N}$ and $X=\mathcal{S}^{\mathbb{Z}}$ with a product measure 
$$\mu=\prod_{n\in\mathbb{Z}}\mu_{n}$$
where
$$\mu_{n}=\left(p_{n}\left(0\right),p_{n}\left(1\right),\dotsc,p_{n}\left(d-1\right)\right)$$
for $n\in\mathbb{Z}$. Suppose that $\mu$ satisfies the Doeblin condition and that the shift $T:\left(X,\mu\right)\to\left(X,\mu\right)$ is nonsingular and conservative. Denote 
$$\lim_{\left|n\right|\to\infty}p_{n}\left(s\right)=q\left(s\right)\text{ for }s\in\mathcal{S}.$$
If the shift is not of type $\text{II}_{1}$ then there exists $\alpha\in\mathcal{S}$ such that without loss of generality
$$\sum_{n\geq1}\left(p_{n}\left(\alpha\right)-q\left(\alpha\right)\right)^{2}=\infty.$$
Take a subsequence $\left(j_{k}:k\geq1\right)$ of positive integers with $j_{k+1}-j_{k}\geq3$ such that 
$s:=\mathrm{sign}\left(p_{j_{k}}\left(\alpha\right)-q\left(\alpha\right)\right)$ is constant in $k\geq1$ and $\sum_{k\geq1}\left(p_{j_{k}}\left(\alpha\right)-q\left(\alpha\right)\right)^{2}=\infty$. It is easy to see that for every $k\geq1$ there exists $\beta_{k}\in\mathcal{S}$ such that
$$\mathrm{sign}\left(p_{j_{k}}\left(\beta_{k}\right)-q\left(\beta_{k}\right)\right)=-s.$$
Then choose a sequence $\left(i_k:k\geq1\right)$ of negative integers with $i_{k}-i_{k+1}\geq3$ and $\mathrm{sign}\left(p_{j_{k}}\left(\beta_{k}\right)-p_{i_{k}}\left(\beta_{k}\right)\right)=s$ is constant for $k\geq1$. Then the sequence of admissible pairs
$$\left(B_{k},B_{k}'\right)\text{ for }B_{k}=\left[\alpha,\alpha,\alpha\right],\,B_{k}'=\left[\alpha,\beta_{k},\alpha\right],\quad k\geq1,$$
satisfies that
$$D_{k}=\log\left(\frac{P_{i_{k}}\left(B_{k}'\right)P_{j_{k}}\left(B_{k}\right)}{P_{i_{k}}\left(B_{k}\right)P_{j_{k}}\left(B_{k}'\right)}\right)\asymp\left(p_{j_{k}}\left(\alpha\right)-p_{i_{k}}\left(\alpha\right)\right)+\left(p_{i_{k}}\left(\beta_{k}\right)-p_{j_{k}}\left(\beta_{k}\right)\right).$$
Then in the same way of the proof of Theorem \ref{Theorem: Theorem D} we conclude that the Bernoulli shift is of type $\mathrm{III}_1$.
\end{rem}

\begin{proof}[Proof of Theorem \ref{Theorem: Theorem E}]

Let $T:\left(X_{\boldsymbol{G}},\mu\right)\to\left(X_{\boldsymbol{G}},\mu\right)$ be a nonsingular and conservative Golden Mean SFT, where $\mu$ is defined by the transition matrices
$$P_{n}=\left(\begin{array}{ccc}
p_{n} & 0 & 1-p_{n}\\
p'_{n} & 0 & 1-p'_{n}\\
0 & 1 & 0
\end{array}\right),\quad n\in\mathbb{Z}.
$$
Assume that $\lim_{\left|n\right|\to\infty}P_{n}=Q$ for
$$Q=\left(\begin{array}{ccc}
q & 0 & 1-q\\
q' & 0 & 1-q'\\
0 & 1 & 0
\end{array}\right).
$$
If the shift is not of type $\mathrm{II}_1$ then without loss of generality we have that
$$\sum_{n\geq1}\left(p_{n}-q\right)^{2}+\sum_{n\geq1}\left(p'_{n}-q'\right)^{2}=\infty.$$

\subsubsection*{Case 1}

Suppose that
$$\sum_{n\geq1}\left(p_{n}-q\right)^{2}=\infty.$$
As the shift is nonsingular, by the same reasoning we used in the fullshift there is a subsequence $I_0\subset\mathbb{N}$ with $\mathrm{sign}\left(p_{n}-q\right)=\mathrm{sign}\left(p_{n+1}-q\right)$ for all $n\in I_0$ such that $\sum_{n\in I_0}\left(p_{n}-q\right)^{2}=\infty$. Let
$$I:=\left\{ n\in I_0:\mathrm{sign}\left(p_{n}-q\right)=\mathrm{sign}\left(p_{n+2}-q\right)\right\} \subset I_{0}.$$
By the nonsingularity of the shift, using Corollary \ref{Corollary: Nonsingularity of the Shift} and the Cauchy-Schwartz inequality we see that
$$\sum_{n\notin I}\left(p_{n}-q\right)^{2}\leq\sum_{n\notin I}\left(p_{n}-p_{n+2}\right)^{2}<\infty.$$
We then get that
$$\sum_{n\in I}\left(p_{n}-q\right)^{2}=\infty$$
on an index set $I\subset\mathbb{N}$ with the property
$$\mathrm{sign}\left(p_{n}-q\right)=\mathrm{sign}\left(p_{n+1}-q\right)=\mathrm{sign}\left(p_{n+2}-q\right),\quad n\in I.$$
Write $s:=\mathrm{sign}\left(p_{n}-q\right)$ for any $n\in I$. We can find a subsequence $\left(j_{k}:k\geq1\right)\subset I$ and a sequence $\left(i_k:k\geq1\right)$ of negative integers with the same properties as in the fullshift. Since the adjacency matrix $\boldsymbol{G}$ of the Golden Mean SFT satisfies $\boldsymbol{G}^M>0$ for $M=3$ and the admissible pairs we will find will have length $L=4$, we further require that $j_k-j_{k-1}\geq7$ and $i_k-i_{k+1}\geq7$ for all $k\geq1$. This can be done using the same considerations as in the fullshift. Consider the admissible pairs
$$\left(B_{0},B_{0}'\right)\text{ for }B_{0}=\left[0,0,0,0\right],\,B_{0}'=\left[0,2,1,0\right],$$
and
$$\left(B_{1},B_{1}'\right)\text{ for }B_{1}=\left[0,0,0,2\right],\,B_{1}'=\left[0,2,1,2\right].$$
For every $k\geq1$ we have that
\begin{align*}
D_{0,k}
&:=\log\left(\frac{P_{i_{k}}\left(B_{0}'\right)P_{j_{k}}\left(B_{0}\right)}{P_{i_{k}}\left(B_{0}\right)P_{j_{k}}\left(B_{0}'\right)}\right)\\
&=\log\frac{p_{j_{k}}}{p_{i_{k}}}+\log\frac{p_{j_{k}+1}}{p_{i_{k}+1}}+\log\frac{1-p_{i_{k}}}{1-p_{j_{k}}}+\log\frac{p_{j_{k}+2}}{p_{i_{k}+2}}+\log\frac{p'_{i_{k}+2}}{p'_{j_{k}+2}}
\end{align*}
and that
\begin{align*}
D_{1,k}
&:=\log\left(\frac{P_{i_{k}}\left(B_{1}'\right)P_{j_{k}}\left(B_{1}\right)}{P_{i_{k}}\left(B_{1}\right)P_{j_{k}}\left(B_{1}'\right)}\right)\\
&=\log\frac{p_{j_{k}}}{p_{i_{k}}}+\log\frac{p_{j_{k}+1}}{p_{i_{k}+1}}+\log\frac{1-p_{i_{k}}}{1-p_{j_{k}}}+\log\frac{1-p_{j_{k}+2}}{1-p_{i_{k}+2}}+\log\frac{1-p'_{i_{k}+2}}{1-p'_{j_{k}+2}}.
\end{align*}
Define for $k\geq1$,
$$g\left(k\right)=\begin{cases}
0 & \mathrm{sign}\left(\log\frac{p'_{i_{k}+2}}{p'_{j_{k}+2}}\right)=s\\
1 & \mathrm{sign}\left(\log\frac{1-p'_{i_{k}+1}}{1-p'_{j_{k}+1}}\right)=s
\end{cases}.$$
Similarly to the fullshift, by the approximation in \ref{Fact: Log Approximation} we have that
$$\log\frac{p_{j_{k}}}{p_{i_{k}}}\asymp\log\frac{p_{j_{k}+1}}{p_{i_{k}+1}}\asymp p_{j_{k}}-p_{i_{k}};$$
we also have that
$$\mathrm{sign}\left(\log\frac{1-p_{i_{k}}}{1-p_{j_{k}}}\right)=\mathrm{sign}\left(\log\frac{p_{j_{k}+1}}{p_{i_{k}+1}}\right)=s,\quad k\geq1;$$
and by the nonsingularity of the shift
$$\log\frac{1-p_{i_{k}}}{1-p_{j_{k}}}+\log\frac{1-p_{j_{k}+2}}{1-p_{i_{k}+2}}\asymp \left(p_{j_k}-p_{j_k+2}\right)+\left(p_{i_k+2}-p_{i_k}\right)$$
is a square-summable sequence. Then the very same proof of Claim \ref{Claim: Convergent Fullshift} shows that the sequence $D_k:=D_{g\left(k\right),k}$, $k\geq1$, satisfies
$$D_k\xrightarrow[k\to\infty]{}0\quad\text{and }\sum_{k\geq1}D_k^2=\infty.$$
Thus, the admissible configuration $\left(B_{g\left(k\right)}\left(i_k\right),B_{g\left(k\right)}'\left(j_k\right)\right)$, $k\geq1$, satisfies the conditions of Lemma \ref{Lemma: Convergent Case} and we conclude that $\mathrm{e}\left(\mathscr{\mathscr{R}}_{\boldsymbol{G}}\right)=\mathbb{R}$.

\subsubsection*{Case 2}

Suppose that
$$\sum_{n\geq1}\left(p_{n}-q\right)^{2}<\infty.$$
Then by Theorem \ref{Theorem: Kakutani Criteria} we can assume that $p_n=q$ for all $n\geq1$ without changing the equivalence class of the measure. Note that in this case we have that $\sum_{n\geq1}\left(p'_{n}-q'\right)^{2}=\infty$. Choose sequences $\left(i_k:k\geq1\right)$ and $\left(j_k:k\geq1\right)$ in the same way we chose in the first case. Consider the admissible pair
$$\left(B,B'\right)\text{ for }B=\left[1,0,0,2\right],\,B'=\left[1,2,1,2\right]$$
for which
\begin{align*}
D_{k}
&:=\log\left(\frac{P_{i_{k}}\left(B'\right)P_{j_{k}}\left(B\right)}{P_{i_{k}}\left(B\right)P_{j_{k}}\left(B'\right)}\right)\\
&=\log\frac{1-p'_{i_{k}}}{1-p'_{j_{k}}}+\log\frac{p'_{j_{k}}}{p'_{i_{k}}}+\log\frac{1-p'_{i_{k}+2}}{1-p'_{j_{k}+2}}+\log\frac{q}{p_{i_{k}+1}}+\log\frac{1-q}{1-p_{i_{k}+2}}.
\end{align*}
We can choose the sequence $\left(i_k:k\geq1\right)$ such that $p_{i_k}\xrightarrow[k\to\infty]{}q$ fast enough so that
$$\sum_{k\geq1}\left(\log\frac{q}{p_{i_{k}+1}}\right)^{2}+\sum_{k\geq1}\left(\log\frac{1-q}{1-p_{i_{k}+2}}\right)^{2}<\infty,$$
using the nonsingularity of the shift as in Corollary \ref{Corollary: Nonsingularity of the Shift}. Then since
$$\mathrm{sign}\left(p'_{i_{k}}-p'_{j_{k}}\right)=\mathrm{sign}\left(p'_{i_{k}+1}-p'_{j_{k}+1}\right)=\mathrm{sign}\left(p'_{i_{k}+2}-p'_{j_{k}+2}\right)$$
is constant for $k\geq1$ it follows that $\sum_{k\geq1}D_{k}^{2}=\infty$. It is clear that $D_k\xrightarrow[k\to\infty]{}0$ so that $\left(B\left(i_k\right),B'\left(j_k\right)\right)$, $k\geq1$, is an admissible configuration that satisfies the conditions of Lemma \ref{Lemma: Convergent Case} so that $\mathrm{e}\left(\mathscr{\mathscr{R}}_{\boldsymbol{G}}\right)=\mathbb{R}$.
\end{proof}

\section{Examples}

We introduce a construction of a class of Markov SFT's for which the shift is nonsingular and conservative, providing a class of examples to our Theorems \ref{Theorem: Theorem D} and \ref{Theorem: Theorem E}. This method of construction is due to Kosloff, first introduced in \cite{kosloff2018manifolds} and then was used in \cite{kosloff2014conservativeanosov} to construct conservative Anosov diffeomorphisms of the $2$-dimensional torus without a Lebesgue a.c.i.m. In the original construction further effort has been made to show that the shift is of type $\mathrm{III}_1$. Using our results, to conclude that the shift is of type $\mathrm{III}_1$ we only need to show that the shift is nonsingular and conservative and that the measure is not equivalent to a homogeneous Markov measure.

Throughout the construction we fix some $0<q<1$. To define a Markov measure on $X=\left\{ 0,1\right\} ^{\mathbb{Z}}$ or on the Golden Mean SFT $X_{\boldsymbol{G}}\subset\left\{ 0,1,2\right\} ^{\mathbb{Z}}$, we take an input of the form $\left\{ p_k,M_{k},N_{k}:k\geq1\right\}$, where $1>p_k\xrightarrow[k\to\infty]{}1/2$ and $M_{k}$ and $N_{k}$ are positive integers satisfying
$$1=M_{0}<N_{1}<M_{1}\dotsm<M_{k-1}<N_{k}<M_{k}<\dotsc.$$ 
For such input denote the stochastic matrices
$$Q_k=\left(\begin{array}{cc}
p_k & 1-p_k\\
q & 1-q
\end{array}\right)
,\quad
Q=\left(\begin{array}{cc}
1/2 & 1/2\\
q & 1-q
\end{array}\right)$$
for the fullshift $X$, and the stochastic matrices
$$Q_k=\left(\begin{array}{ccc}
p_k & 0 & 1-p_k\\
q & 0 & 1-q\\
0 & 1 & 0
\end{array}\right)
,\quad
Q=\left(\begin{array}{ccc}
1/2 & 0 & 1/2\\
q & 0 & 1-q\\
0 & 1 & 0
\end{array}\right)$$
for the Golden Mean SFT $X_{\boldsymbol{G}}$. Then let $\mu_{\left\{ p_k,M_{k},N_{k}:k\geq1\right\} }$ be the Markov measure with transition matrices $\left(P_{n}:n\in\mathbb{Z}\right)$ defined by
$$P_{n}=\begin{cases}
Q_k & n\in\left[M_{k-1},N_{k}\right)\\
Q & \text{otherwise}
\end{cases},\text{ for }n\in\mathbb{Z},$$
with an appropriate definition of the transition matrices $Q_k$ depending on whether we consider $X$ or $X_{\boldsymbol{G}}$. The coordinate distributions $\left(\pi_{n}:n\in\mathbb{Z}\right)$ can be chosen arbitrarily as long as they satisfy the consistency condition $\pi_{n}P_{n}=\pi_{n+1}$ for all $n\in\mathbb{Z}$. This defines a Markov measure with the Doeblin condition on $X$ or on $X_{\boldsymbol{G}}$ that fits the convergent scenario we discussed in Theorems \ref{Theorem: Theorem D} and \ref{Theorem: Theorem E}.

\begin{rem}
Note that the particular choice of $q$ that has been made in \cite{kosloff2018manifolds} was designed only to compute essential values, which we do not need, hence we consider an arbitrary $q$. Further note that in the Golden Mean SFT it is obvious that the Markov measure $\mu_{\left\{ p_k,M_{k},N_{k}:k\geq1\right\} }$ is not equivalent to a non-homogeneous product measure. But also in the fullshift, as
$$P_{n}\xrightarrow[n\to\infty]{}\left(\begin{array}{cc}
1/2 & 1/2\\
q & 1-q
\end{array}\right),$$
by choosing $q\neq1/2$ and using Corollary \ref{Corollary: Equivalence to Homogeneous} we ensure that $\mu_{\left\{ p_k,M_{k},N_{k}:k\geq1\right\} }$ can not be equivalent to any product measure. Thus, this construction provides a class examples for Markov fullshift and Markov SFT with Markov measures that are not equivalent to product measures and are of type $\mathrm{III}_1$.
\end{rem}

\begin{prop}
\label{Proposition: Nonsingularity+II1}
Let $\mu=\mu_{\left\{ p_k,M_{k},N_{k}:k\geq1\right\} }$ be as above, either for the fullshift or for the Golden Mean SFT. Then we have the following.
\begin{enumerate}
	\item The shift is nonsingular with respect to $\mu$ if, and only if,
\begin{equation}
\label{eq:16}
\sum_{k\geq1}\left(p_k-1/2\right)^{2}<\infty.
\end{equation}
	\item The only homogeneous Markov measure that can be equivalent to $\mu$ is the Markov measure $\nu$ defined by the transition matrix $Q$, and $\mu$ is \emph{not} equivalent to $\nu$ if, and only if,
\begin{equation}
\label{eq:17}
\sum_{k\geq1}\left(N_{k}-M_{k-1}\right)\left(p_k-1/2\right)^{2}=\infty.
\end{equation}
\end{enumerate}
\end{prop}

\begin{proof}
For the nonsingularity we see that
$$\sum_{n\geq1}\sum_{s,t\in\mathcal{S}}\left(\sqrt{P_{n}\left(s,t\right)}-\sqrt{P_{n-1}\left(s,t\right)}\right)^{2}=\sum_{k\geq1}\sum_{s,t\in\mathcal{S}}\left(\sqrt{Q_{k}\left(s,t\right)}-\sqrt{Q\left(s,t\right)}\right)^{2},$$
and that
\begin{equation}
\label{eq:18}
\left(\sqrt{Q_{k}\left(s,t\right)}-\sqrt{Q\left(s,t\right)}\right)^{2}\asymp\left(p_{k}-1/2\right)^{2}
\end{equation}
for all $s,t\in\mathcal{S}$, so the criteria follows from Corollary \ref{Corollary: Nonsingularity of the Shift}. For the equivalence of $\mu$ to $\nu$, we see that
\begin{align*}
&\sum_{n\geq1}\sum_{s,t\in\mathcal{S}}\left(\sqrt{P_{n}\left(s,t\right)}-\sqrt{Q\left(s,t\right)}\right)^{2}\\
&\qquad\qquad\qquad=\sum_{k\geq1}\sum_{s,t\in\mathcal{S}}\left(N_{k}-M_{k-1}\right)\left(\sqrt{Q_{k}\left(s,t\right)}-\sqrt{Q\left(s,t\right)}\right)^{2}.
\end{align*}
Then by the approximation \eqref{eq:18} the criteria follows from Corollary \ref{Corollary: Equivalence to Homogeneous}.
\end{proof}

The following proposition was proved by Kosloff as part of \cite[Theorem 13]{kosloff2018manifolds} for the Golden Mean SFT. It is based on a calculation that can be carried out for both the fullshift and the Golden Mean SFT.

\begin{prop}[Kosloff]
\label{Proposition: Conservativeness}
Let $\mu=\mu_{\left\{ p_k,M_{k},N_{k}:k\geq1\right\} }$ be as above, either for the fullshift or for the Golden Mean SFT. Let us denote $\lambda\left(p\right)=\frac{p}{1-p}$ for $p\in\left(0,1\right)$. Then the shift is conservative with respect to $\mu$ if
\begin{equation}
\label{eq:19}
1<\lambda\left(p_{k}\right)^{M_{k-1}}\leq\mathrm{e}^{r_{k}}\text{ where }\sum_{k\geq1}r_{k}<\infty,
\end{equation}
and
\begin{equation}
\label{eq:20}
\sum_{k\geq1}\left(M_{k}-2N_{k}\right)\lambda\left(p_{1}\right)^{-N_{k}}=\infty.
\end{equation}
\end{prop}

Note that with condition \eqref{eq:19} we have $\left(p_{k}-1/2\right)^{2}\asymp\left(1-\lambda\left(p_{k}\right)\right)^{2}=o\left(r_{k}^{2}\right)$ as well as $\sum_{k\geq1}r_{k}<\infty$, so this implies condition \eqref{eq:16} for the nonsingularity. The conditions of Propositions \ref{Proposition: Nonsingularity+II1} and \ref{Proposition: Conservativeness} can be fulfilled simultaneously for various choices of $\left\{ p_k,M_{k},N_{k}:k\geq1\right\}$ and it may be constructed inductively as follows. Fix some sequence $\left(r_{k}:k\geq1\right)$ of positive numbers with $\sum_{k\geq1}r_{k}<\infty$. Let an arbitrary $1/2<p_1<1$ and arbitrary positive integers $M_{0}<N_{1}<M_{1}$. Assume that $\left\{ p_j,M_{j},N_{j}:1\leq j\leq k\right\}$  has been defined for some $k\in\mathbb{N}$. Define $p_{k+1}$, $N_{k+1}$ and $M_{k+1}$ as follows.
\begin{enumerate}
	\item choose $p_{k+1}$ such that $1<\lambda\left(p_{k+1}\right)^{M_{k}}\leq\text{e}^{r_k}$;
	\item choose $N_{k+1}>M_{k}$ such that $\left(N_{k+1}-M_{k}\right)\left(1-\lambda\left(p_{k+1}\right)\right)^{2}\geq 1$; and,
	\item choose $M_{k+1}>N_{k+1}$ such that $\left(M_{k+1}-2N_{k+1}\right)\lambda\left(p_{1}\right)^{-N_{k+1}}\geq 1$.
\end{enumerate}\

The following corollary then follows immediately from our Theorems \ref{Theorem: Theorem D} and \ref{Theorem: Theorem E}.

\begin{cor}
If $\mu=\mu_{\left\{ p_k,M_{k},N_{k}:k\geq1\right\} }$ satisfies conditions \eqref{eq:17} as well as conditions \eqref{eq:19} and \eqref{eq:20}, then $\mu$ is a Markov measure which is not equivalent to any product measure, and the shift is nonsingular, conservative and of type $\mathrm{III}_1$ with respect to $\mu$.
\end{cor}

The construction of a conservative divergent Markov measure is more subtle. Yet, the class of conservative divergent Markov measures is non empty as there is the construction of such Bernoulli shift due to Kosloff \cite{kosloff2014}.

\begin{appendix}

\section{Mixing Properties of Markov Measures}
\label{Appendix: Mixing Properties of Markov Measures}

Let $\left(X_A,\mu\right)$ be a topologically-mixing MSFT that satisfies the Doeblin condition \ref{Doeblin}. The sequence of transition matrices of $\mu$ will be denoted by $\left(P_n:n\in\mathbb{Z}\right)$ and the coordinates distributions by $\left(\pi_n:n\in\mathbb{Z}\right)$. The integer $M$ will stand for the first positive integer for which $A^M>0$ and the constant $\delta$ is the positive constant of the Doeblin condition.

\begin{lem}
\label{Lemma: Mixing1}
The marginals of $\mu$ satisfy
$$\delta^{M}\leq\pi_{n}\left(s\right)\leq1-\delta^{M},\quad n\in\mathbb{Z}, s\in\mathcal{S}.$$
Also for every $N\geq1$ sufficiently large, specifically $N\geq M$, we have that
$$\delta^{M}\leq P^{\left(n,n+N\right)}\left(s,t\right)\leq1-\delta^{M},\quad n\in\mathbb{Z}, s,t\in\mathcal{S}.$$
\end{lem}

\begin{proof}
We start by bounding the transition matrices. It is an immediate observation that
   $$P^{\left(n,n+M-1\right)}\left(s,t\right)\geq\delta^{M},\quad n\in\mathbb{Z}, s,t\in\mathcal{S}.$$   
Thus, for any $N\geq M$,
\begin{align*}
P^{\left(n,n+N\right)}\left(s,t\right)
& = \sum_{u\in\mathcal{S}}P^{\left(n,n+N-M\right)}\left(s,u\right)P^{\left(n+N-M+1,n+N\right)}\left(u,t\right)\\
& \geq \sum_{u\in\mathcal{S}}P^{\left(n,n+N-M\right)}\left(s,u\right)\delta^{M}=\delta^{M},\quad n\in\mathbb{Z}, s,t\in\mathcal{S},
\end{align*}
Now we easily deduce the same bound for the coordinate distributions:
\begin{align*}
\pi_{n}\left(s\right)
& = \sum_{t\in\mathcal{S}}\pi_{n-M}\left(t\right)P^{\left(n-M,n-1\right)}\left(t,s\right)\\
& \geq \sum_{t\in\mathcal{S}}\pi_{n-M}\left(t\right)\delta^{M}=\delta^{M},\quad n\in\mathbb{Z},s\in\mathcal{S}.
\end{align*}
Those lower bounds yield the upper bounds so the proof is complete.
\end{proof}

\begin{lem}
\label{Lemma: Mixing2}
There exists a constant $C\left(\delta,M\right)\in\left(0,1\right)$ depending only on $M$ and $\delta$, such that for every pair of Borel sets $E\in\sigma\left(\dotsc,X_{n-1},X_n\right)$ and $F\in\sigma\left(X_m,X_{m+1},\dotsc\right)$, if $m-n\geq M$ then
$$C\left(\delta,M\right)\mu\left(E\right)\mu\left(F\right)\leq\mu\left(E\cap F\right)\leq C\left(\delta,M\right)^{-1}\mu\left(E\right)\mu\left(F\right).$$
\end{lem}

\begin{proof}
Observe that for $E\in\sigma\left(\dotsc,X_{n-1},X_n\right)$ and $m$ with $m-n\geq M$,
\begin{align*}
\mu\left(E\cap\left\{ X_{m}=s\right\} \right)
& = \sum_{t\in\mathcal{S}}\mu\left(E\cap\left\{ X_{n+1}=t\right\} \right)\mu\left(\left\{ X_{m}=s\right\} \mid E\cap\left\{ X_{n+1}=t\right\} \right)\\
& =\sum_{t\in\mathcal{S}}\mu\left(E\cap\left\{ X_{n+1}=t\right\} \right)\mu\left(\left\{ X_{m}=s\right\} \mid\left\{X_{n+1}=t\right\} \right)\\
& = \sum_{t\in\mathcal{S}}\mu\left(E\cap\left\{ X_{n+1}=t\right\} \right)P^{\left(n+1,m-1\right)}\left(t,s\right)\\
& \geq\delta^{M}\sum_{t\in\mathcal{S}}\mu\left(E\cap\left\{X_{n+1}=t\right\} \right)=\delta^{M}\mu\left(E\right),\quad s\in\mathcal{S},
\end{align*}
where we used the Markov property and the lower bound of Lemma \ref{Lemma: Mixing1}. If instead we use the upper bound of Lemma \ref{Lemma: Mixing1} we get that
$$\mu\left(E\cap\left\{X_{m}=s\right\} \right)\leq\left(1-\delta^{M}\right)\mu\left(E\right),\quad s\in\mathcal{S}.$$
Now we conclude:
\begin{align*}
\mu\left(E\cap F\right)
& = \sum_{s\in\mathcal{S}}\mu\left(E\cap\left\{X_{m-1}=s\right\} \right)\mu\left(F\mid E\cap\left\{X_{m-1}=s\right\} \right)\\
& = \sum_{s\in\mathcal{S}}\mu\left(E\cap\left\{X_{m-1}=s\right\} \right)\mu\left(F\mid \left\{X_{m-1}=s\right\}\right)\\
& \geq \delta^{M}\mu\left(E\right)\sum_{s\in\mathcal{S}}\mu\left(F\mid \left\{X_{m-1}=s\right\}\right)\\
& = \delta^{M}\mu\left(E\right)\sum_{s\in\mathcal{S}}\frac{\mu\left(F\cap\left\{X_{m-1}=s\right\} \right)}{\pi_{m-1}\left(s\right)}\\
& \geq \frac{\delta^{M}}{1-\delta^{M}}\mu\left(E\right)\sum_{s\in\mathcal{S}}\mu\left(F\cap\left\{X_{m-1}=s\right\} \right)\\
& = \frac{\delta^{M}}{1-\delta^{M}}\mu\left(E\right)\mu\left(F\right),
\end{align*}
where we used the Markov property, the above observations and the lower bound of Lemma \ref{Lemma: Mixing1}. A similar use of the upper bound of Lemma \ref{Lemma: Mixing1} shows that
$$\mu\left(E\cap F\right)\leq\frac{1-\delta^{M}}{\delta^{M}}\mu\left(E\right)\mu\left(F\right).$$
Then the Lemma holds for the constant $C\left(\delta,M\right)=\delta^{M}/\left(1-\delta^{M}\right)>0$.
\end{proof}

\section{A Criteria for Equivalence of Markov Measures}
\label{Appendix: Criteria for Equivalence of MM}

Let $\mathcal{S}$ be a finite state space and let $X_A$ be a topologically-mixing SFT in $\mathcal{S}^{\mathbb{Z}}$. Let $\nu$ and $\mu$ be a pair of Markov measures on $X_{A}$ defined by $\left(\pi_n,P_n:n\in\mathbb{Z}\right)$ and $\left(\lambda_n,Q_n:n\in\mathbb{Z}\right)$, respectively. Recall that $\left(\pi_n,\widehat{P}_n:n\in\mathbb{Z}\right)$ is the sequence of the reversed sequence of transitions of $\nu$,
$$\widehat{P}_{n}\left(s,t\right)=\nu\left(X_{n-1}=t\mid X_{n}=s\right)=\frac{\pi_{n-1}\left(t\right)}{\pi_{n}\left(s\right)}P_{n-1}\left(t,s\right),\quad s,t\in\mathcal{S},n\in\mathbb{Z}.$$
Then $\pi_n\widehat{P}_n=\pi_{n-1}$ for $n\in\mathbb{Z}$. Clearly, a Markov measure that is specified by $\left(P_n:n\in\mathbb{Z}\right)$ with the usual convention of the dependence direction is also specified by $\left(\widehat{P}_n:n\in\mathbb{Z}\right)$ for the reversed dependence direction.

For the sake of completeness we repeat the formulation of Theorem \ref{Theorem: Kakutani Criteria}.

\begin{thm}
Let $\nu$ and $\mu$ be Markov measures on a topologically-mixing SFT $X_A$ defined by $\left(P_n:n\in\mathbb{Z}\right)$ and $\left(Q_n:n\in\mathbb{Z}\right)$, respectively. Suppose that  both satisfy the Doeblin condition \ref{Doeblin}. Then $\nu\ll\mu$ if, and only if,
$$\sum_{n\geq1}\sum_{s,u,v,t\in\mathcal{S}}\mathrm{d}_{n}^{2}\left[\nu,\mu\right]\left(s,u,v,t\right)<\infty,$$
where for $n\geq1$ and $s,t,u,v\in\mathcal{S}$ we denote the numbers
$$\mathrm{d}_{n}^{2}\left[\nu,\mu\right]\left(s,u,v,t\right):=\left(\sqrt{\widehat{P}_{-n}\left(u,s\right)P_{n}\left(v,t\right)}-\sqrt{\widehat{Q}_{-n}\left(u,s\right)Q_{n}\left(v,t\right)}\right)^{2}.$$
\end{thm}

In order to prove Theorem \ref{Theorem: Kakutani Criteria} we use a significant generalization of the Kakutani dichotomy, established by Kabanov--Lipcer--Shiryaev (see \cite{shiryaev1978absolute} and references therein). Their theorem is formulated in the following setting. Let $X$ be a standard Borel space and fix $\left(\mathcal{A}_n:n\geq1\right)$ a filtration of $X$. Let $\nu$ and $\mu$ be probability measures on $X$. For every $n\geq1$ let $\nu_n$ and $\mu_n$ be the restriction of $\nu$ and $\mu$ to $\mathcal{A}_n$, respectively, and suppose that $\nu_n\ll\nu_n$ for all $n\geq1$. Let 
$$\mathrm{m}_{n}\left(x\right):=\frac{d\nu_{n}}{d\mu_{n}}\left(x\right)\text{ for }n\geq1.$$
Then $\left(\mathrm{m}_{n}:n\geq1\right)$ is a martingale with respect to the natural filtration and it satisfies 
$$\intop_{X}\mathrm{m}_{n}\left(x\right)d\mu\left(x\right)=1$$ for every $n\geq1$, so by the martingale convergence theorem
$$\mathrm{m}_{\infty}\left(x\right):=\lim_{n\to\infty}\mathrm{m}_{n}\left(x\right)\text{ exists for }\mu\text{-a.e. } x\in X.$$
In fact, this limit exists also for $\nu$-a.e. $x\in X$. See also \cite[Chapter~6]{shiryaev2013p} for a comprehensive representation.

\begin{thm}[Kabanov--Lipcer--Shiryaev]
\label{Theorem: KLS}
In the above setting, let
$$\mathrm{M}_{n}\left(x\right):=\mathrm{m}_{n}\left(x\right)\mathrm{m}_{n-1}^{-1}\left(x\right),\text{ where }\mathrm{M}_{n}\left(x\right):=0\text{ if }\mathrm{m}_{n-1}\left(x\right)=0.$$
Then
$$\nu\ll\mu\iff\sum_{n\geq1}\left(1-\mathbf{E}_{\mu}\left(\sqrt{\mathrm{M}_{n}\left(x\right)}\mid\mathcal{A}_{n-1}\right)\right) <\infty \text{ for }\nu\text{-a.e. }x\in X.$$
\end{thm}

Consider a topologically-mixing SFT $X_A$ and let $\left(\mathcal{A}_n:n\geq1\right)$ be the natural filtration defined by $\mathcal{A}_n=\sigma\left(X_k:\left|k\right|\leq n\right)$. Let $\nu$ and $\mu$ be Markov measures on $X_A$ defined by $\left(P_n:n\in\mathbb{Z}\right)$ and $\left(Q_n:n\in\mathbb{Z}\right)$, respectively. Suppose that both $\nu$ and $\mu$ satisfy the Doeblin condition. The following result was established in \cite{lodkin1971absolute, lepage1975likelihood} and we deduce it from Theorem \ref{Theorem: KLS} by a straightforward calculation.

\begin{cor}
\label{Corollary: KLS}
In the above setting we have
\begin{equation}
\label{eq:5}
\nu\ll\mu\iff\sum_{n\geq1}\sum_{s,t\in\mathcal{S}}\mathrm{d}_{n}^{2}\left[\nu,\mu\right]\left(s,x_{-\left(n-1\right)},x_{n-1},t\right)<\infty\text{ for }\nu\text{-a.e. }x\in X_A,
\end{equation}
where for $n\geq1$ and $s,t\in\mathcal{S}$ we denote the random variables
\begin{align*}
&\mathrm{d}_{n}^{2}\left[\nu,\mu\right]\left(s,x_{-\left(n-1\right)},x_{n-1},t\right)\\
&\qquad\quad:=\left(\sqrt{\widehat{P}_{-n}\left(u,x_{-\left(n-1\right)}\right)P_{n}\left(x_{n-1},t\right)}-\sqrt{\widehat{Q}_{-n}\left(x_{-\left(n-1\right)},s\right)Q_{n}\left(x_{n-1},t\right)}\right)^{2}.
\end{align*}
\end{cor}

\begin{proof}
It is clear that every such $\nu$ and $\mu$ satisfy $\nu_n\ll\mu_n$ for all $n\geq1$. Let us calculate $\mathbf{E}_{\mu}\left[\sqrt{M_n}\mid\mathcal{A}_{n-1}\right]$. The Radon--Nikodym derivatives are given by
$$\mathrm{m}_{n}\left(x\right):=\frac{d\nu_{n}}{d\mu_{n}}\left(x\right)=\frac{\pi_{-n}\left(x_{-n}\right)}{\lambda_{-n}\left(x_{-n}\right)}\prod_{i=-n}^{n-1}\frac{P_{i}\left(x_{i},x_{i+1}\right)}{Q_{i}\left(x_{i},x_{i+1}\right)},$$
so one can see that
$$\mathrm{M}_{n}\left(x\right)=\frac{\widehat{P}_{-\left(n-1\right)}\left(x_{-n},x_{-\left(n-1\right)}\right)P_{n-1}\left(x_{n-1},x_{n}\right)}{\widehat{Q}_{-\left(n-1\right)}\left(x_{-n},x_{-\left(n-1\right)}\right)Q_{n-1}\left(x_{n-1},x_{n}\right)}.$$
It follows that
\begin{align*}
& \mathbf{E}_{\mu}\left(\sqrt{\mathrm{M}_{n}\left(x\right)}\mid\mathcal{A}_{n-1}\right)\\
& \quad=\sum_{s,t\in\mathcal{S}} \sqrt{\frac{\widehat{P}_{-\left(n-1\right)}\left(s,x_{-\left(n-1\right)}\right)P_{n-1}\left(x_{n-1},t\right)}{\widehat{Q}_{-\left(n-1\right)}\left(s,x_{-\left(n-1\right)}\right)Q_{n-1}\left(x_{n-1},t\right)}}\widehat{Q}_{-\left(n-1\right)}\left(s,x_{-\left(n-1\right)}\right)Q_{n-1}\left(x_{n-1},t\right)\\
&\quad=\sum_{s,t\in\mathcal{S}}\sqrt{\widehat{P}_{-\left(n-1\right)}\left(s,x_{-\left(n-1\right)}\right)P_{n-1}\left(x_{n-1},t\right)\widehat{Q}_{-\left(n-1\right)}\left(s,x_{-\left(n-1\right)}\right)Q_{n-1}\left(x_{n-1},t\right)}\\
&\quad=1-\frac{1}{2}\sum_{s,t\in\mathcal{S}}\left(\sqrt{\widehat{P}_{-\left(n-1\right)}\left(s,x_{-\left(n-1\right)}\right)P_{n-1}\left(x_{n-1},t\right)}\right.\\
&\quad\qquad\qquad\qquad\qquad\qquad\qquad\left.-\sqrt{\widehat{Q}_{-\left(n-1\right)}\left(s,x_{-\left(n-1\right)}\right)Q_{n-1}\left(x_{n-1},t\right)}\right)^{2}.
\end{align*}
This shows that
$$1-\mathbf{E}_{\mu}\left(\sqrt{M_{n}\left(x\right)}\mid\mathcal{A}_{n}\right)=\frac{1}{2}\sum_{s,t\in\mathcal{S}}\mathrm{d}_{n}^{2}\left[\nu,\mu\right]\left(s,x_{-\left(n-1\right)},x_{n-1},t\right),\quad n\geq1,$$
and from Theorem \ref{Theorem: KLS} the proof is complete.
\end{proof}

It follows from Corollary \ref{Corollary: KLS} that if $\sum_{n\geq1}\sum_{s,t,u,v\in\mathcal{S}}\mathrm{d}_{n}^{2}\left[\nu,\mu\right]\left(s,u,v,t\right)<\infty$ then $\sum_{n\geq1}\sum_{s,t\in\mathcal{S}}\mathrm{d}_{n}^{2}\left[\nu,\mu\right]\left(s,x_{-\left(n-1\right)},x_{n-1},t\right)<\infty$ for \emph{every} $x\in X_A$ so that $\nu\ll\mu$. Thus, to prove Theorem \ref{Theorem: Kakutani Criteria} we need to show that the right-hand side of condition \eqref{eq:5} does not hold if $\sum_{n\geq1}\sum_{s,t,u,v\in\mathcal{S}}\mathrm{d}_{n}^{2}\left[\nu,\mu\right]\left(s,u,v,t\right)=\infty$. For this we use the following probabilistic lemma.

\begin{lem}
\label{Lemma: Probabilistic Lemma}
Let $\left(a_{n}:n\geq1\right)$ be a sequence of non-negative numbers such that $\sum_{n\geq1}a_{n}=\infty$. Let $\left(\Omega,\mathbb{P}\right)$ be a probability space and let $\left(A_{n}:n\geq1\right)$ be a sequence of events in $\Omega$. Then
$$\mathbb{P}\left(\sum_{n\geq1}a_{n}\mathbf{1}_{A_{n}}=\infty\right)\geq\liminf_{n\to\infty}\mathbb{P}\left(A_{n}\right).$$
\end{lem}

\begin{proof}
Denote $p:=\liminf_{n\to\infty}\mathbb{P}\left(A_{n}\right)$. Excluding trivialities assume $p>0$ and let $0<\epsilon<p$ be arbitrary. Fix $N\geq1$ such that $\mathbb{P}\left(A_{n}\right)\geq\epsilon$ for every $n\geq N$. For every $C>0$ let
$$F_{C}:=\left\{ \sum_{n\geq N}a_{n}\mathbf{1}_{A_{n}}\leq C\right\},$$
and suppose toward a contradiction that $\mathbb{P}\left(F_{C}\right)>1-\epsilon$ for some $C>0$. Then
$$\mathbb{P}\left(F_{C}\right)-\mathbb{P}\left(A_{n}^{\mathsf{c}}\right)\geq\mathbb{P}\left(F_{C}\right)-\left(1-\epsilon\right)>0,\quad n\geq N.$$
We then have
\begin{align*}
\infty
&=\sum_{n\geq N}a_{n}\left(\mathbb{P}\left(F_{C}\right)-\left(1-\epsilon\right)\right)\\
&\leq\sum_{n\geq N}a_{n}\left(\mathbb{P}\left(F_{C}\right)-\mathbb{P}\left(A_{n}^{\mathsf{c}}\right)\right)\\
&\leq \sum_{n\geq N}a_{n}\mathbb{P}\left(F_{C}\cap A_{n}\right)=\sum_{n\geq N}a_{n}\mathbf{E}\left[{\mathbf{1}_{F_{C}}\mathbf{1}_{A_{n}}}\right]\\
&=\mathbf{E}\left[\mathbf{1}_{F_{C}}\sum_{n\geq N}a_{n}\mathbf{1}_{A_{n}}\right]\\
&\leq C\mathbb{P}\left(F_{C}\right),
\end{align*}
where the second inequality is general: $\mathbb{P}\left(E\cap F\right)\geq\mathbb{P}\left(E\right)-\mathbb{P}\left(F^{\mathsf{c}}\right)$, the next equality is by monotone convergence and the last one is by the definition of $F_C$. This is a contradiction so that $\mathbb{P}\left(F_{C}\right)\leq1-\epsilon$ for all $C\geq 0$. It follows that
$$\mathbb{P}\left(\sum_{n\geq1}a_{n}\mathbf{1}_{A_{n}}<\infty\right)=\mathbb{P}\left(\sum_{n\geq N}a_{n}\mathbf{1}_{A_{n}}<\infty\right)\leq\limsup_{C\to\infty}\mathbb{P}\left(F_{C}\right)\leq1-\epsilon.$$
Since $0<\epsilon<p$ is arbitrary the proof is complete.
\end{proof}

\begin{proof}[Proof of Theorem \ref{Theorem: Kakutani Criteria}]
One of the implications is immediate from Corollary \ref{Corollary: KLS} as we already mentioned. For the other implication suppose that
$$\sum_{n\geq1}\sum_{s,t,u,v\in\mathcal{S}}\mathrm{d}_{n}^{2}\left[\nu,\mu\right]\left(s,u,v,t\right)=\infty.$$
Then $\sum_{n\geq1}\sum_{s,t\in\mathcal{S}}\mathrm{d}_{n}^{2}\left[\nu,\mu\right]\left(s,u_0,v_0,t\right)=\infty$ for some $u_0,v_0\in\mathcal{S}$. Consider the sets
$$A_n:=\left\{X_{-\left(n-1\right)}=u_0,X_{n+1}=v_0\right\},\quad n\geq1.$$
By the topologically-mixing, the Doeblin condition and Proposition \ref{Proposition: Mixing in MSFT},
$$\liminf_{n\to\infty}\nu\left(A_{n}\right)\geq C\left(\delta,M\right)\liminf_{n\to\infty}\pi_{-\left(n-1\right)}\left(u\right)\pi_{n+1}\left(v\right)\geq C\left(\delta,M\right)\delta^{2}>0.$$
Then by the Lemma \ref{Lemma: Probabilistic Lemma} we conclude that
\begin{align*}
&\sum_{n\geq1}\sum_{s,t\in\mathcal{S}}\mathrm{d}_{n}^{2}\left[\nu,\mu\right]\left(s,x_{-\left(n-1\right)},x_{n-1},t\right)\\
&\qquad=\sum_{n\geq1}\left(\sum_{s,t,u,v\in\mathcal{S}}\mathrm{d}_{n}^{2}\left[\nu,\mu\right]\left(s,u,v,t\right)\right)\mathbf{1}_{A_{n}}\left(x\right)=\infty
\end{align*}
on a $\nu$-positive measure set (of measure at least $C\left(\delta,M\right)\delta^{2}$). Then by condition \eqref{eq:5} $\nu$ is not absolutely continuous with respect to $\mu$.
\end{proof}

\end{appendix}

\addtocontents{toc}{\protect\setcounter{tocdepth}{-1}}

\bibliographystyle{acm}
\bibliography{References}

\nocite{*}

\end{document}